\documentclass[11pt,twoside, leqno]{article}
\usepackage{amssymb}
\usepackage{amsmath}
\usepackage{mathrsfs}
\usepackage{amsthm}
\usepackage{amsfonts}
\usepackage{enumerate}
\usepackage{esint}
\usepackage{latexsym,amssymb}

\allowdisplaybreaks

\def\ls{\lesssim}
\def\gs{\gtrsim}
\def\fz{\infty}

\def\r{\right}
\def\lf{\left}

\def\supp{{\mathop\mathrm{\,supp\,}}}

\def\rr{{\mathbb R}}

\def\rn{{{\rr}^n}}
\def\zz{{\mathbb Z}}
\def\nn{{\mathbb N}}
\def\cc{{\mathbb C}}

\newcommand{\wz}{\widetilde}
\newcommand{\oz}{\overline}

\newcommand{\cs}{{\mathcal S}}

\def\az{\alpha}
\def\lz{\lambda}

\def\bdz{\Delta}

\def\dz {\delta}

\def\bz{\beta}

\def\fai{\varphi}
\def\gz{{\gamma}}
\def\bgz{{\Gamma}}
\def\vz{\varphi}
\def\tz{\theta}
\def\sz{\sigma}

\def\wz{\widetilde}

\def\ls{\lesssim}
\def\gs{\gtrsim}

\def\oz{\omega}

\def\pat{\partial}

\def\hs{\hspace{0.3cm}}

\def\dsum{\displaystyle\sum}
\def\dint{\displaystyle\int}
\def\dfrac{\displaystyle\frac}
\def\dsup{\displaystyle\sup}

\newtheorem{theorem}{Theorem}[section]
\newtheorem{lemma}[theorem]{Lemma}
\newtheorem{corollary}[theorem]{Corollary}
\newtheorem{proposition}[theorem]{Proposition}
\theoremstyle{definition}
\newtheorem{remark}[theorem]{Remark}
\newtheorem{definition}[theorem]{Definition}
\numberwithin{equation}{section}
\def\supp{{\mathop\mathrm{\,supp\,}}}

\def\lfz{\lfloor}
\def\rfz{\rfloor}
\numberwithin{equation}{section}

\begin{document}

\arraycolsep=1pt

\title{\bf\Large Riesz Transform Characterizations of
Musielak-Orlicz-Hardy Spaces
\footnotetext {\hspace{-0.35cm}
2010 {\it Mathematics Subject Classification}. Primary: 47B06;
Secondary: 42B20, 42B30, 42B35, 46E30.
\endgraf {\it Key words and phrases}.
Riesz transform, harmonic function, Cauchy-Riemann equation,
Musielak-Orlicz-Hardy space.
\endgraf Jun Cao is supported by the
Fundamental Research Funds for the Central Universities (Grant No. 2012YBXS16).
Der-Chen Chang is partially supported by an NSF grant DMS-1203845 and Hong Kong
RGC competitive earmarked research grant $\#$601410. Dachun Yang is supported by the National
Natural Science Foundation  of China (Grant Nos. 11171027 and 11361020) and
the Specialized Research Fund for the Doctoral Program of Higher Education
of China (Grant No. 20120003110003).}}
\author{Jun Cao, Der-Chen Chang, Dachun Yang\footnote{Corresponding author}\ \ and Sibei Yang}
\date{}
\maketitle

\vspace{-0.8cm}

\begin{center}
\begin{minipage}{13.5cm}
{\small {\bf Abstract}. Let $\varphi$ be a Musielak-Orlicz function satisfying that, for any
$(x,\,t)\in\mathbb{R}^n\times(0,\,\infty)$,
$\varphi(\cdot,\,t)$ belongs to the Muckenhoupt weight class $A_\infty (\mathbb{R}^n)$
with the critical weight exponent $q(\varphi)\in[1,\,\infty)$ and $\varphi(x,\,\cdot)$
is an Orlicz function with $0<i(\varphi)\le I(\varphi)\le 1$ which are, respectively,
its critical lower type and upper type. In this article, the authors establish the
Riesz transform characterizations of the Musielak-Orlicz-Hardy spaces $H_\varphi
(\mathbb{R}^n)$ which are generalizations of weighted Hardy spaces and Orlicz-Hardy spaces. Precisely,
the authors characterize $H_\varphi (\mathbb{R}^n)$ via all the first order Riesz transforms
when $\frac{i(\varphi)}{q(\varphi)}>\frac{n-1}{n}$, and via all the Riesz transforms
with the order not more than $m\in\mathbb{N}$ when $\frac{i(\varphi)}{q(\varphi)}>\frac{n-1}{n+m-1}$.
Moreover, the authors also establish the Riesz transform characterizations of
$H_\varphi(\mathbb{R}^n)$, respectively, by means of the
higher order Riesz transforms defined via the homogenous harmonic polynomials
or the odd order Riesz transforms. Even if when $\varphi(x,t):=tw(x)$ for all
$x\in{\mathbb R}^n$ and $t\in [0,\infty)$, these results also widen the range
of weights in the known Riesz characterization
of the classical weighted Hardy space $H^1_w({\mathbb R}^n)$
obtained by R. L. Wheeden from $w\in A_1({\mathbb R}^n)$
into $w\in A_\infty({\mathbb R}^n)$ with the sharp range
$q(w)\in [1,\frac n{n-1})$, where $q(w)$ denotes the critical
index of the weight $w$.}
\end{minipage}
\end{center}

\section{Introduction\label{s1}}
\hskip\parindent
Denote by $\cs(\rn)$ the \emph{space of all Schwartz functions} on $\rn$.
For $j\in\{1,\,\ldots,\,n\}$, $f\in \mathcal{S}(\rn)$ and $x\in\rn$,
the $j$-\emph{th Riesz transform} of $f$ is usually defined by
\begin{eqnarray}\label{eqn RZ}
R_j(f)(x):=\lim_{\epsilon\to 0^+} C_{(n)} \dint_{\{y\in\rn:\ |y|>\epsilon\}}
\frac{y_j}{|y|^{n+1}}f(x-y)\,dy,
\end{eqnarray}
here and hereafter, $\epsilon\to 0^+$ means that $\epsilon>0$ and $\epsilon\to 0$,
$C_{(n)}:=\frac{\bgz((n+1)/{2})}{\pi^{(n+1)/2}}$ and $\bgz$ denotes
the Gamma function. As a natural generalization of the Hilbert transform
to the Euclidean space of higher dimension, Riesz transforms may be the
most  typical examples of Calder\'on-Zygmund operators which have been
extensively studied by many mathematicians
(see, for example, \cite{St70,St91,Gr08} and their references).

While most literatures on Riesz transforms focus on their boundedness
on various function spaces, the main purpose of this article is to establish
Riesz transform characterizations of some Hardy spaces of Musielak-Orlicz type.
This research originates from Fefferman-Stein's 1972 celebrating seminal paper
\cite{FS72} and was then extended by Wheeden to the weighted Hardy space $H^1_w(\rn)$
(see \cite{Wh76}). It is known that, when establishing Riesz transform
characterizations of Hardy spaces $H^p(\rn)$, we need to extend the elements
of $H^p(\rn)$ to the upper half space $\rr^{n+1}_+:=\rn\times (0,\,\fz)$
via the Poisson integral.
This extension in turn has close relationship with the analytical
definition of $H^p(\rn)$ which is the key starting point of studying the Hardy space,
before people paid attention to the real-variable theory of $H^p(\rn)$
(see \cite{SW60,St61,SW68,MW78,Wh79}). Recall also that
the real-variable theory of $H^p(\rn)$ and their weighted versions plays
very important roles in analysis such as harmonic analysis and partial
differential equations; see, for example, \cite{St91,Gr09, DHHR11}.

The Riesz transform characterization of Hardy spaces on $\rn$ is one
of the most important and useful real-variable characterizations
(see \cite{FS72,MW78,Wh79,St91,Uc84}).
Indeed it is well known that Riesz transforms have many interesting properties.
For example, they are the simplest, non-trivial, ``invariant" operators
under the acting of the group of rotations in the Euclidean space $\rn$,
and they also constitute typical and important examples of Fourier multipliers.
Moreover, they can be used to mediate between various combinations
of partial derivatives of functions. All these properties make Riesz
transforms ubiquitous in mathematics (see \cite{St70} for more details
on their applications). Recall also that Riesz transforms
are not bounded on Lebesgue spaces $L^p(\rn)$ when $p\in(0,\,1]$.
One of the main motivations to introduce the
Hardy space $H^p(\rn)$ with $p\in(0,\,1]$ is to find a suitable substitute
of $L^p(\rn)$ when studying the
boundedness of some operators.

Denote by $\cs'(\rn)$ the \emph{dual space} of $\cs(\rn)$ (namely, the \emph{space of all
tempered distributions}). Let $f\in \mathcal{S}'(\rn)$.
Recall that a distribution $f\in \mathcal{S}'(\rn)$
is called a \emph{distribution restricted at infinity},
if there exists a positive number $r$ sufficiently large such that,
for all $\phi\in \mathcal{S}(\rn)$,  $f*\phi\in L^r(\rn)$.
Fefferman and Stein \cite{FS72} proved the following important result
(see also \cite[p.\,123,\ Proposition 3]{St91}
for a more detailed description).

\begin{theorem}[\cite{FS72}]\label{thm RTC of HP}
Let $p\in(\frac{n-1}{n},\,\fz)$, $\phi\in\mathcal{S}(\rn)$ satisfy $\int_{\rn}\phi(x)\,dx=1$
and $f$ be a distribution restricted at infinity. Then
$f\in H^p(\rn)$ if and only if there exists a positive constant
$A$ such that, for all $\epsilon\in(0,\,\fz)$ and $j\in\{1,\,\ldots,\,n\}$,
$f*\phi_\epsilon,\ R_j(f)*\phi_\epsilon\in L^p(\rn)$ and
\begin{eqnarray*}
\lf\|f*\phi_\epsilon\r\|_{L^p(\rn)}+
\dsum_{j=1}^n \lf\|R_j(f)*\phi_\epsilon\r\|_{L^p(\rn)}\le A,
\end{eqnarray*}
where, for all $x\in\rn$, $\phi_\epsilon(x):=\frac{1}{\epsilon^n}\phi(\frac{x}{\epsilon})$.
 Moreover, there exists a positive constant $C$,
independent of $f$ and $\epsilon$, such that
\begin{eqnarray*}
\frac{1}{C}\|f\|_{H^p(\rn)}\le A\le C\|f\|_{H^p(\rn)}.
\end{eqnarray*}
\end{theorem}

It is known that, for $p\in (0,\frac{n-1}{n}]$, $H^p(\rn)$ can
be characterized no longer by first order Riesz transforms
but by higher order Riesz transforms
(see \cite[p.\,168]{FS72} or
Theorem \ref{thm2} below for more details).

In this article, we establish the Riesz transform characterization of the
Musielak-Orlicz-Hardy space $H_\varphi (\mathbb{R}^n)$
which is introduced by Ky \cite{K11}.
It is known that the space $H_\varphi (\mathbb{R}^n)$ is a generalization of the
Orlicz-Hardy space introduced by Str\"omberg \cite{Str79} and Janson \cite{Ja80},
and the weighted Hardy space
$H_w^p(\mathbb{R}^n)$ for $w\in A_\infty(\rn)$ and $p\in(0,\,1]$, introduced by
Garc\'ia-Cuerva \cite{G79} and Str\"omberg-Torchinsky \cite{S-T89}. Here, $A_q(\rn)$ with $q\in[1,\fz]$ denotes
the class of Muckenhoupt weights (see, for example,
\cite{G79,GR85,Gr08} for their definitions and properties).
Moreover, the space $H_\varphi (\mathbb{R}^n)$ also has already found
many applications in analysis (see,
for example, \cite{bfg10,bgk12,hyy,Ky13,K11} and their references).

Recall that, in \cite{Wh76}, Wheeden characterized the weighted Hardy
space $H^1_w(\rn)$ via first order Riesz transforms when $w\in A_1(\rn)$.
Our results extend the corresponding results of \cite{FS72,Wh76} essentially;
see Remark \ref{r1.1} below for more details.

In order to state the main results of this article, let us recall some necessary
definitions and notation.

Let $\fai$ be a nonnegative function on $\rn\times[0,\,\fz)$.
The function $\fai$ is called a \emph{Musielak-Orlicz function},
if, for any $x\in\rn$, $\fai(x,\,\cdot)$ is an Orlicz function on $[0,\,\fz)$
and, for any $t\in[0,\,\fz)$, $\fai(\cdot,\,t)$ is measurable on $\rn$.
Here a function $\Phi:[0,\,\fz)\to[0,\,\fz)$ is called an \emph{Orlicz function}, if it
is nondecreasing, $\Phi(0)=0$, $\Phi(t)>0$ for $t\in(0,\,\fz)$ and
$\lim_{t\to\fz}\Phi(t)=\fz$ (see, for example,
\cite{M83}). Remark that, unlike the usual case, \emph{such a $\Phi$ may not
be convex}.

For an Orlicz function $\Phi$, the most useful tool to study its growth property may be  the
upper and the lower types of $\Phi$. More precisely, for $p\in(0,\,\fz)$,
a function $\Phi$ is said to be of
\emph{upper} (resp. \emph{lower}) \emph{type} $p$, if there exists a positive constant
$C$ such that, for all $s\in[1,\fz)$ (resp. $s\in[0,1]$) and $t\in[0,\fz)$,
\begin{equation}\label{1.x1}
\Phi(st)\le Cs^p \Phi(t).
\end{equation}

Let $\fai$ be a Musielak-Orlicz function.
The \emph{Musielak-Orlicz space $L^{\fai}(\rn)$},
which was first introduced by Musielak \cite{M83},
is defined to be the set of all measurable functions $f$ such that
$\int_{\rn}\fai(x,|f(x)|)\,d x<\fz$ with the \emph{Luxembourg-Nakano (quasi-)norm}:
\begin{eqnarray}\label{1.4}
\|f\|_{L^{\fai}(\rn)}:=\inf\lf\{\lz\in(0,\fz):\ \int_{\rn}
\fai\lf(x,\frac{|f(x)|}{\lz}\r)\,dx\le1\r\}.
\end{eqnarray}

We also need the following notion of Muckenhoupt weight classes
from \cite{Mu72}.
For $q\in(1,\,\fz)$, a nonnegative locally integrable function $w$ on $\rn$ is said to
belong to $A_q(\rn)$, if, for all balls $B\subset \rn$,
\begin{eqnarray*}
\lf\{\frac{1}{|B|}\dint_{B}
w(x)\,dx\r\}\lf\{\frac{1}{|B|}\dint_{B}
\lf[w(x)\r]^{1-q'}\,dx\r\}^{q-1}\le [w]_{A_q(\rn)}<\fz,
\end{eqnarray*}
here and hereafter, $q':=\frac{q}{q-1}$ denotes the \emph{conjugate exponent} of $q$.
Moreover, the nonnegative locally integrable function $w$ is said to belong to $A_1(\rn)$,
if, for all balls $B\subset \rn$,
\begin{eqnarray*}
\lf\{\frac{1}{|B|}\dint_{B} w(x)\,dx\r\}
\lf\{\mathop\mathrm{ess\,sup}_{y\in B}[w(y)]^{-1}\r\}\le [w]_{A_1(\rn)}<\fz.
\end{eqnarray*}
Let $A_\fz(\rn):=\cup_{q\in[1,\,\fz)}A_q(\rn)$.
Moreover, throughout the whole article, we \emph{always assume} that the Musielak-Orlicz functions
satisfy the following \emph{growth assumptions} (see \cite[Definition 2.1]{K11}).

\medskip

\noindent {\bf Assumption ($\fai$).}
Let $\fai:\ \rn\times[0,\fz)\to[0,\fz)$ be a Musielak-Orlicz function satisfying
the following two conditions:
\vspace{-0.25cm}
\begin{enumerate}
 \item[(i)] for any $t\in(0,\,\fz)$,
 $\fai(\cdot,\,t)\in A_{\fz}(\rn)$;
\vspace{-0.25cm}
\item[(ii)] there exists $p\in(0,\,1]$ such that,
for every $x\in\rn$, $\fai(x,\,\cdot)$ is of upper type $1$ and of
lower type $p$.
\end{enumerate}

Notice that there exist many examples of functions satisfying Assumption
$(\fai)$. For example,  for all $(x,\,t)\in\rn\times[0,\,\fz)$,
$\fai(x,\,t):=\oz(x)\Phi(t)$ satisfies Assumption $(\fai)$ if
$\oz\in A_{\fz}(\rn)$ and $\Phi$ is an Orlicz function of lower
type $p$ for some $p\in(0,\,1]$ and upper type 1. A typical example
of such an Orlicz function $\Phi$ is $\Phi(t):=t^p$, with $p\in(0,\,1]$, for all $t\in [0,\,\fz)$;
see, for example, \cite{hyy,Ky13,K11} for more examples.
Another typical example of functions satisfying Assumption
$(\fai)$ is $\fai(x,t):=\frac{t^{\az}}{[\ln(e+|x|)]^{\bz}+[\ln(e+t)]^{\gz}}$
for all $x\in\rn$ and $t\in[0,\,\fz)$ with any $\az\in(0,\,1]$ and
$\bz,\,\gz\in [0,\,\fz)$ (see \cite{K11} for further details).

For a Musielak-Orlicz  function $\fai$ satisfying Assumption $(\fai)$,
the following critical indices are useful. Let
\begin{eqnarray}\label{1.5}
I(\fai):=\inf\lf\{p\in(0,\,\fz):\ \text{for any}\ x\in\rn,\ \fai(x,\,\cdot)\ \text{is of upper
type}\ p\r.\\ \nonumber
\hspace{2cm} \lf.\text{with}\ C\ \text{as\ in}\ \eqref{1.x1}\ \text{independent\ of}\ x\r\},
\end{eqnarray}
\begin{eqnarray}\label{1.6}
i(\fai):=\sup\{p\in(0,\,\fz):\ \text{for any $x\in\rn$},
\ \fai(x,\,\cdot)\ \text{is of lower type}\ p\\ \nonumber
\hspace{2cm} \text{with}\ C\ \text{as\ in}\ \eqref{1.x1}\ \text{independent\ of}\ x\}
\end{eqnarray}
and
\begin{eqnarray}\label{1.7}
\hs q(\fai)&&:=\inf\lf\{q\in[1,\,\fz):\ \text{for any $t\in(0,\,\fz)$},\
\fai(\cdot,\,t)\in A_{q}(\rn)\r.\\
&&\hspace{2cm}\lf.\nonumber
\text{with}\ [\fai(\cdot,\,t)]_{A_q(\rn)}\ \text{independent of}\ t\r\}.
\end{eqnarray}

Let $\nn:=\{1,\,2,\, \ldots\}$ and $\zz_+:=\{0\}\cup\nn$. For any
$\tz:=(\tz_{1},\ldots,\tz_{n})\in\zz_{+}^{n}$, let
$|\tz|:=\tz_{1}+\cdots+\tz_{n}$ and
$\partial^{\tz}_x:=\frac{\partial^{|\tz|}}{\partial
{x_{1}^{\tz_{1}}}\cdots\partial {x_{n}^{\tz_{n}}}}$.
For $m\in\nn$, define
$$\cs_m(\rn):=\lf\{\phi\in\cs(\rn):\ \sup_{x\in\rn}\sup_{
\bz\in\zz^n_+,\,|\bz|\le m+1}(1+|x|)^{(m+2)(n+1)}|\partial^\bz_x\phi(x)|\le1\r\}.
$$
Then, for all $x\in\rn$ and $f\in\cs'(\rn)$, the \emph{non-tangential grand maximal function}
$f^\ast_m$ of $f$ is defined by setting,
$$f^\ast_m(x):=\sup_{\phi\in\cs_m(\rn)}\sup_{|y-x|<t,\,t\in(0,\fz)}|f\ast\phi_t(y)|,
$$
where, for all $t\in(0,\,\fz)$, $\phi_t(\cdot):=t^{-n}\phi(\frac{\cdot}{t})$.
When $m(\fai):=\lfz n[q(\fai)/i(\fai)-1]\rfz$,
where $q(\fai)$ and $i(\fai)$ are, respectively,  as in \eqref{1.7} and \eqref{1.6},
and $\lfz s\rfz$ for $s\in\rr$ denotes the \emph{maximal integer not more than $s$},
we \emph{denote $f^\ast_{m(\fai)}$ simply by $f^\ast$}.

Ky \cite{K11} introduced the following Musielak-Orlicz-Hardy spaces $H_\fai(\rn)$.

\begin{definition}[\cite{K11}]\label{def MOH space}
Let $\fai$ satisfy Assumption $(\fai)$.
The \emph{Musielak-Orlicz-Hardy space $H_\fai(\rn)$} is defined to
be the space of all $f\in\cs'(\rn)$ such that $f^\ast\in L^\fai(\rn)$
with the \emph{quasi-norm}
$\|f\|_{H_\fai(\rn)}:=\|f^\ast\|_{L^\fai(\rn)}$.
\end{definition}

\begin{remark}\label{rem defMOH}
(i) We point out that, if $\fai(x,\,t):=w(x)t^p$, with $p\in(0,\,1]$
and $w\in A_\infty(\rn)$,  for all $(x,\,t)\in\rn\times[0,\,\fz)$,
the {Musielak-Orlicz-Hardy space $H_\fai(\rn)$}
coincides with the weighted Hardy space $H_w^p(\rn)$ studied in \cite{G79,S-T89};
if $\fai(x,\,t):=\Phi(t)$, with $\Phi$ an Orlicz function whose
upper type is $1$ and lower type $p\in (0,1]$, for all $(x,\,t)\in\rn\times[0\,,\fz)$,
$H_\fai(\rn)$ coincides with the Orlicz-Hardy space
$H_\Phi(\rn)$ introduced in \cite{Ja80, Str79}.
Also, the Musielak-Orlicz-Hardy space $H_\fai(\rn)$ has
proved useful in the study of other analysis problems
when we take various different Musielak-Orlicz functions $\fai$
(see, for example, \cite{bfg10,Ky13,K11}).

(ii) For all $(x,\,t)\in\rn\times [0,\,\fz)$, let
\begin{eqnarray}\label{eqn deffai}
\wz\fai(x,\,t):=\dint_0^t\frac{\fai(x,\,s)}{s}\,ds.
\end{eqnarray}
It is easy to see that $\wz\fai$ is strictly increasing and continuous in
$t$. Similar to \cite[Proposition 3.1]{Vi87}, we know that $\wz\fai$ inherits the
types of $\fai$ and is equivalent to $\fai$, which implies that
$H_\fai(\rn)=H_{\wz\fai}(\rn)$
with equivalent quasi-norms. Thus, without loss of generality,
in the remainder of this article,
we may \emph{always assume} that $\fai(x,\cdot)$ for all $x\in\rn$
is strictly increasing and continuous on $[0,\,\fz)$.

(iii) Let $\overline{H_\fai(\rn)\cap L^2(\rn)}^{\|\cdot\|_{H_\fai(\rn)}}$
be the \emph{completion} of the set $H_\fai(\rn)\cap L^2(\rn)$ under the
quasi-norm $\|\cdot\|_{H_\fai(\rn)}$. From the fact that $H_\fai(\rn)\cap L^2(\rn)$
is dense in $H_\fai(\rn)$ which is a simple corollary of \cite[Theorem 3.1]{K11}
in the case $q=\fz$, we immediately deduce that
\begin{eqnarray*}
\overline{H_\fai(\rn)\cap L^2(\rn)}^{\|\cdot\|_{H_\fai(\rn)}}=H_\fai(\rn).
\end{eqnarray*}
\end{remark}

In order to obtain Riesz transform characterizations of $H_\fai(\rn)$,
we have to overcome some essential difficulties, which have already existed even in the case
of weighted Hardy spaces $H_w^p(\rn)$, caused by weights. One of the most typical difficulties
relies on the fact that, for an arbitrary $f\in H_\fai(\rn)$,
we cannot obtain directly that $f$ is a distribution restricted at
infinity as in the unweighted case.
To be more precise, let $\phi\in \mathcal{S}(\rn)$ with $\int_\rn \phi(x)\,dx=1$,
$p\in(0,\,i(\fai))$, $t\in(0,\,\fz)$ and $x\in\rn$, assume that $|f*\phi_t(x)|\ge 1$;
then, following Stein's argument (see \cite[pp.\,100-101]{St91}) and using the lower type
 $p$ property of $\fai(\cdot,t)$, we see that
\begin{eqnarray}\label{eqn difficult}
\lf|f*\phi_t(x)\r|^p\ls \frac{\int_{B(x,\,1)}\lf|f*\phi_t(x)\r|^p\fai(y,\,1)\,dy}
{\int_{B(x,\,1)}\fai(y,\,1)\,dy}\ls \|f\|^p_{H_\fai(\rn)}\frac{1}
{{\int_{B(x,\,1)}\fai(y,\,1)\,dy}}.
\end{eqnarray}
From this, it follows that, in order to show that $f*\phi_t\in L^\fz(\rn)$, 
we need $\frac{1} {{\int_{B(x,\,1)}\fai(y,\,1)\,dy}}$ is uniformly bounded in $x$
(see \cite[Remark 3.3]{B90} for a similar condition in the case of
weighted Hardy spaces).

To get rid of this unpleasant and awkward restriction, we aptly adapt a smart and wise strategy that has
recently been used in the case of Hardy spaces associated with operators (see, for example,
\cite[Theorem 5.2]{HMM11} for Riesz transform characterizations
of Hardy spaces associated with second order divergence form elliptic operators).
Precisely, we first restrict the \emph{working space} to $H_\fai(\rn)\cap L^2(\rn)$,
in which the Riesz transforms and Poisson integrals are well defined, then
we extend the working space by a process of completion via the quasi-norm
based on Riesz transforms.   In particular, we introduce the following
Riesz Musielak-Orlicz-Hardy space.
\begin{definition}\label{def RMOHS}
Let $\fai$ satisfy Assumption $(\fai)$.
The \emph{Riesz Musielak-Orlicz-Hardy space $H_{\fai,\,\mathrm{Riesz}}(\rn)$} is defined to
be the completion of the set
\begin{eqnarray*}
\mathbb{H}_{\fai,\,\mathrm{Riesz}}(\rn)
:=\{f\in L^2(\rn):\ \|f\|_{H_{\fai,\,\mathrm{Riesz}}(\rn)}<\fz\}
\end{eqnarray*}
under the \emph{quasi-norm} $\|\cdot\|_{H_{\fai,\,\mathrm{Riesz}}
(\rn)}$, where, for all $f\in L^2(\rn)$,
\begin{eqnarray*}
\lf\|f\r\|_{H_{\fai,\,\mathrm{Riesz}}(\rn)}:=\|f\|_{L^\fai(\rn)}+
\dsum_{j=1}^n\lf\|R_j(f)\r\|_{L^\fai(\rn)}.
\end{eqnarray*}
\end{definition}

Now we give out the first main result of this article.

\begin{theorem}\label{thm1}
Let $\fai$ satisfy Assumption $(\fai)$ and $\frac{i(\fai)}{q(\fai)}\in(\frac{n-1}{n},\,\fz)$
with $i(\fai)$ and $q(\fai)$ as in \eqref{1.6} and \eqref{1.7}, respectively.
Then $H_\fai(\rn)=H_{\fai,\,\mathrm{Riesz}}(\rn)$ with equivalent quasi-norms.
\end{theorem}

\begin{remark}\label{r1.1}
(i) We point out that, if $\fai(x,\,t):=t^p$, with $p\in(\frac{n-1}{n},\,1]$,
for all $(x,\,t)\in\rn\times[0,\,\fz)$, then the difficulty in \eqref{eqn difficult}
disappear automatically. Thus, in this case,
there is no need to use the restriction to $L^2(\rn)$. Observe that, in this case,
the range $p\in (\frac{n-1}{n},\,1]$ in Theorem \ref{thm1} coincides
with the range of $p$ in Theorem \ref{thm RTC of HP} obtained
by Fefferman and Stein \cite{FS72}, which is the known best possible.
Moreover, compared with Theorem \ref{thm RTC of HP},
an advantage of Theorem \ref{thm1} is that, in Theorem \ref{thm1}, we
do not assume the a priori assumption that $f$ is a distribution restricted at infinity.

(ii) Recall that Wheeden  in
\cite{Wh76} characterized $H^1_w(\rn)$,  with $w\in A_1(\rn)$, by the first order
Riesz transforms, which corresponds to the case when $\varphi(x,t):=tw(x)$ for all
$x\in{\mathbb R}^n$ and $t\in [0,\infty)$ of Theorem \ref{thm1}; even in
this special case, Theorem \ref{thm1} also widens the range
of weights from $w\in A_1({\mathbb R}^n)$ into $w\in A_\fz({\mathbb R}^n)$
with the sharp range $q(w)\in [1,\frac n{n-1})$,
where $q(w)$ denotes the critical index of the weight $w$ as in \eqref{1.7}.
Moreover, if we let $\fai(x,\,t):=w(x)t^p$, with
$w\in\mathrm{A}_\fz(\rn)$ and $p\in(\frac{q(w)(n-1)}{n},\,1]$,
for all $(x,\,t)\in\rn\times[0,\fz)$, then $\vz$ also satisfies the assumptions
of Theorem \ref{thm1} and Theorem \ref{thm1} with this $\vz$
extends the results obtained by Wheeden in
\cite{Wh76} from the case $p=1$ into the case $p<1$.

(iii) In the sense of (i) and (ii) of this remark,
the range of $\frac{i(\fai)}{q(\fai)}\in(\frac{n-1}{n},\fz)$ in Theorem \ref{thm1}
is the best possible for the first order Riesz transform characterization
of $H_\fai(\rn)$.
\end{remark}

As in the case of $H^p(\rn)$, the proof of Theorem \ref{thm1} depends on the
delicate characterizations of $H_\fai(\rn)$ via some harmonic functions and vectors
defined on the upper half space $\rr^{n+1}_+$. To this end,
we first introduce the Musielak-Orlicz-Hardy spaces $H_\fai(\rr^{n+1}_+)$
of harmonic functions (see Definition \ref{def MOH space} below) and
$\mathcal{H}_\fai(\rr^{n+1}_+)$ of harmonic vectors
(see Definition \ref{def MOHVH space} below).  However, unlike in the
unweighted case, we cannot obtain the isomorphisms among
$H_\fai(\rn)$, $H_\fai(\rr^{n+1}_+)$ and $\mathcal{H}_\fai(\rr^{n+1}_+)$
without additional assumptions on $\fai$.
To remedy this, we introduce two subspace spaces,
$H_{\fai,\,2}(\rr^{n+1}_+)$ and $\mathcal{H}_{\fai,\,2}(\rr^{n+1}_+)$,
respectively, of $H_\fai(\rr^{n+1}_+)$ and $\mathcal{H}_\fai(\rr^{n+1}_+)$ (see Definitions
\ref{def MOHH space p>1} and \ref{def MOHVH space} for their definitions).
Then we establish the isomorphisms among
$H_\fai(\rn)$, $H_{\fai,\,2}(\rr^{n+1}_+)$ and $\mathcal{H}_{\fai,\,2}(\rr^{n+1}_+)$
(see Theorem \ref{cor equiv 3spaces} below).

With these results as preparation, let us sketch the proof of Theorem \ref{thm1}.
To prove the inclusion $H_{\fai,\,\mathrm{Riesz}}(\rn)\subset H_\fai(\rn)$, for any
$f\in \mathbb{H}_{\fai,\,\mathrm{Riesz}}(\rn)$, we construct
a generalized Cauchy-Riemann system via the (conjugate) Poisson
integrals of $f$ and $R_j(f)$, which is proved to be in $\mathcal{H}_{\fai,\,2}(\rr^{n+1}_+)$.
This, together with the isomorphism between $H_\fai(\rn)$ and
$\mathcal{H}_{\fai,\,2}(\rr^{n+1}_+)$, shows
$H_{\fai,\,\mathrm{Riesz}}(\rn)\subset H_\fai(\rn)$.

To prove the inverse inclusion, we only need the radial maximal function
characterization of  $H_\fai(\rn)$ and the boundedness of Riesz
transforms on $H_\fai(\rn)$ (see Proposition \ref{pro radial mc} and
Corollary \ref{cor bd of Riesz} below). Here, to prove Corollary
\ref{cor bd of Riesz}, we  establish an interpolation
of operators on weighted Hardy spaces (see Proposition \ref{pro interpolation} below),
which might be useful in establishing the boundedness of other important
operators on $H_\fai(\rn)$ (see Corollary \ref{cor bd of CZ} below).
We should mention that it is also possible to show Corollary \ref{cor bd of Riesz}
directly via the atomic and the molecular characterizations of $H_\fai(\rn)$,
respectively, in \cite[Theorem 1.1]{K11} and \cite[Theorem 4.13]{hyy}.
However, the approach used in this article brings us more useful byproducts
which have wide applications (see Proposition \ref{pro interpolation} and its applications below).

We now turn to the study of higher order Riesz transform characterizations of $H_\fai(\rn)$.
Recall that there are several different approaches to introduce the higher order Riesz transforms
(see, for example, \cite{Ku04}). In the present article, we focus on two kinds of
higher order Riesz transforms: {\bf i)} the higher order Riesz transforms which are compositions
of first order Riesz transforms; {\bf ii)} the higher order Riesz transforms defined via
homogenous harmonic polynomials.

We start with the first one. Here, to simplify the notation,
we restrict ourselves to $H_\fai(\rn)\cap L^2(\rn)$.

\begin{theorem}\label{thm2}
Let $m\in\nn\cap [2,\,\fz)$ and $\fai$ satisfy Assumption $(\fai)$
with $\frac{i(\fai)}{q(\fai)}>\frac{n-1}{n+m-1}$, where $i(\fai)$ and $q(\fai)$ are
as in \eqref{1.6} and \eqref{1.7}, respectively. Assume further that $f\in L^2(\rn)$.
Then $f\in H_\fai(\rn)$ if and only if there exists a positive
constant $A$ such that, for all $k\in\{1,\,\ldots,\,m\}$ and $\{j_1,\,\ldots,\,j_k\}\subset\{1,\,\ldots,\,n\}$,
$f$, $R_{j_1}\cdots R_{j_k}(f)\in L^\fai(\rn)$
and
\begin{eqnarray}\label{1.9x}
\lf\|f\r\|_{L^\fai(\rn)}+
\dsum_{k=1}^m \dsum_{j_1,\,\ldots,\,j_k=1}^n \lf\|R_{j_1}
\cdots R_{j_k}(f)\r\|_{L^\fai(\rn)}\le A.
\end{eqnarray}
Moreover, there exists a positive constant $C$, independent of $f$,
such that
\begin{eqnarray}\label{1.9}
\frac{1}{C} \lf\| f\r\|_{H_\fai(\rn)}&&\le A\le C \lf\| f\r\|_{H_\fai(\rn)}.
\end{eqnarray}
\end{theorem}

\begin{remark}\label{r1.2}
(i) Let $m,\,k\in\nn$ and $\{j_1,\,\ldots,\,j_m\}\subset\{0,\,\ldots,\,n\}$
satisfy that the number of the non-zero elements in  $\{j_1,\,\ldots,\,j_m\}$ is $k$.
Assume further that $R_0:=I$ is the \emph{identity operator}.
Then, we call $R_{j_1}\cdots R_{j_m}$ a $k$-\emph{order Riesz
transform}. Theorem \ref{thm2} implies that, to obtain the
Riesz transform characterization of $H_\fai(\rn)$ for all $\fai$ satisfying
$\frac{i(\fai)}{q(\fai)}>\frac{n-1}{n+m-1}$, we need all the $k$-order Riesz transforms
for all $k\in\{0,\,\ldots,\,m\}$.

(ii) Compared with the first order Riesz transform characterization
in Theorem \ref{thm1}, the higher order Riesz transform characterization
in Theorem \ref{thm2} does have some advantages. For example,
we can relax the restrictions of $\fai$ on
both the type and the weight assumptions. To be more precise,
by letting $m$ sufficiently large,
one can obtain the Riesz transform characterization of
$H_\fai(\rn)$ for any given $\fai$ satisfying Assumption $(\fai)$.
\end{remark}

The scheme of the proof of Theorem \ref{thm2} is similar to that of Theorem \ref{thm1}.
The main difference is to replace the space $\mathcal{H}_\fai(\rr^{n+1}_+)$ by the
Musielak-Orlicz-Hardy space $\mathcal{H}_{\fai,\,m}(\rr^{n+1}_+)$
of tensor-valued functions (see Definition \ref{def MOTVFH} below), since,
in this case, we have to make use of all
Riesz transforms up to order $m$.

Now, we consider the second kind of higher Riesz transforms from Stein
\cite{St70}. Let $f\in\mathcal{S}(\rn)$, $k\in\nn$ and $\mathcal{P}_k$ be a homogenous
harmonic polynomial of degree $k$. The \emph{Riesz transform of} $f$
\emph{of degree} $k$ \emph{associated with} $\mathcal{P}_k$ is defined by
setting, for all $x\in\rn$,

\begin{eqnarray}\label{1.10x}
\mathcal{R}^{\mathcal{P}_k}(f)(x):=\lim_{\epsilon\to 0^+} \dint_{|y|\ge \epsilon}
\frac{\mathcal{P}_k(y)}{|y|^{n+k}}f(x-y)\,dy.
\end{eqnarray}
For more details on homogenous harmonic polynomials, we refer the reader to
\cite[Section\,3 of Chapter\,3]{St70}.

Furthermore, Kurokawa \cite{Ku04} obtained the following relationships
between two kinds of higher Riesz transforms as above.

\begin{proposition}[\cite{Ku04}]\label{p1.1}
Let $m,\,k\in\nn$ and $\{j_1,\,\ldots,\,j_m\}\subset\{0,\,\ldots,\,n\}$
satisfy that the number of the non-zero elements in  $\{j_1,\,\ldots,\,j_m\}$ is $k$.
Let $f\in L^2(\rn)$. Then, for
each $k$-order Riesz transform $R_{j_1}\cdots R_{j_m}$ as in Remark \ref{r1.2},
there exist $\ell\in\nn$ and a positive constant $C$ such that
\begin{eqnarray*}
R_{j_1}\cdots R_{j_m}(f)=C f+(-1)^{k}\dsum_{j=0}^\ell \mathcal{R}^{\mathcal{P}_j}(f),
\end{eqnarray*}
where $\mathcal{P}_j$ ranges over all
the homogenous harmonic polynomials of degree $k-2j$ and
$\mathcal{R}^{\mathcal{P}_j}$ is the higher order Riesz transform
of degree $k-2j$ associated with $\mathcal{P}_j$ defined as in \eqref{1.10x}.
\end{proposition}

Combining Proposition \ref{p1.1} and Theorem \ref{thm2},
we conclude the following corollary, which establishes the Riesz transform
characterization of $H_\fai(\rn)$ in terms of
higher Riesz transforms defined via homogenous harmonic
polynomials.

\begin{corollary}\label{c1.1}
Let $m\in\nn\cap [2,\,\fz)$, $k\in\{0,\,\ldots,\,m\}$
and $\fai$ satisfy Assumption $(\fai)$
with $\frac{i(\fai)}{q(\fai)}>\frac{n-1}{n+m-1}$,
where $i(\fai)$ and $q(\fai)$ are as in \eqref{1.6}
and \eqref{1.7}, respectively. Suppose that $f\in L^2(\rn)$.
Then $f\in H_\fai(\rn)$ if and only if there exists a positive constant $A$ such that,
for all homogenous
harmonic polynomials $\mathcal{P}_j$ of degree $k$,
$f$, $\mathcal{R}^{\mathcal{P}_j}(f)\in L^\fai(\rn)$
and
\begin{eqnarray*}
\lf\| f\r\|_{L^\fai(\rn)}
+\sum_j\lf\| \mathcal{R}^{\mathcal{P}_j}(f)\r\|_{L^\fai(\rn)}
\le A.
\end{eqnarray*}
Moreover, there exists a positive constant $C$, independent of $f$,
such that
\begin{eqnarray*}
\frac{1}{C} \lf\| f\r\|_{H_\fai(\rn)}&&\le A \le C \lf\| f\r\|_{H_\fai(\rn)},
\end{eqnarray*}
where $\mathcal{P}_j$ ranges over all the homogenous harmonic polynomials
of degree $k$ with $k\in\{0,\,\ldots,\,m\}$.
\end{corollary}

Observe that, in Corollary \ref{c1.1}, we use less Riesz transforms
than Theorem \ref{thm2} to characterize $H_\fai(\rn)$,
since not every polynomial of order $k$ is homogeneous harmonic.
Moreover, there arises a natural question for Riesz transform
characterizations of $H_\fai(\rn)$: in these characterizations, can we
use Riesz transforms as less as possible? This question can not be solved
directly by the methods used to prove Theorems \ref{thm1} and
\ref{thm2}, because the heart of these methods relies on the
subharmonic property of the absolute value of a harmonic vector (resp.
tensor-valued function) satisfying the generalized Cauchy-Riemann equation.
Moreover, to construct such a harmonic vector (resp. tensor-valued function),
we always need all the Riesz transforms up to a fixed order.

Corollary \ref{c1.1} provides a method to solve the above problem via replacing all
Riesz transforms up to order $m$ by Riesz transforms defined via
homogenous harmonic polynomials. Another method is from Uchiyama \cite{Uc84,Uc01},
which avoids the use of the subharmonic property by using the Fourier multiplier
and has a close relationship with the constructive proof of the Fefferman-Stein
decomposition of $\mathop\mathrm{BMO}(\rn)$.

The following theorem establishes the odd order Riesz transform
characterization of $H_\fai(\rn)$ based on the method of Uchiyama.

\begin{theorem}\label{thm3}
Let $k\in\nn$ be odd, $\fai$ satisfy Assumption $(\fai)$ and
$\frac{i(\fai)}{q(\fai)}>\max\{p_0,\,\frac{1}{2}\}$,
where $i(\fai)$, $q(\fai)$ and $p_0$ are, respectively, as in \eqref{1.6},
\eqref{1.7} and  Proposition \ref{pro est MFFM} below.
Let $f\in L^2(\rn)$. Then $f\in H_\fai(\rn)$ if and only if,
for all $\{j_1,\,\ldots,\,j_k\}\subset \{1,\,\ldots,\,n\}$, $f$ and
$R_{j_1}\cdots R_{j_k}(f)\in L^\fai(\rn)$.
Moreover, there exists a positive constant $C$, independent of $f$, such that
\begin{eqnarray}\label{1.10}
\frac{1}{C} \lf\| f\r\|_{H_\fai(\rn)}&&\le \lf\| f\r\|_{L^\fai(\rn)}
+\dsum_{j_1,\,\ldots,\,j_k=1}^n \lf\| R_{j_1}\cdots R_{j_k}(f)\r\|_{L^\fai(\rn)}\\
\nonumber&&\le C \lf\| f\r\|_{H_\fai(\rn)}.
\end{eqnarray}
\end{theorem}

\begin{remark}\label{rem thm3}
(i) We point out that, in \cite{Fe74}, Fefferman conjectured that
``nice" conjugate systems, such as the second order Riesz
transforms, would also give a characterization of $H^1(\rr^2)$.
However, Gandulfo, Garc\'ia-Cuerva and Taibleson \cite{GGT76}
have constructed a counterexample to show that even order Riesz transforms
fail to characterize $H^1(\mathbb{R}^2)$. This justifies the characterization
of $H_\fai(\rn)$ via odd order Riesz transforms.

(ii) Compared with Corollary \ref{c1.1}, in Theorem \ref{thm3}, we use much less
Riesz transforms. However, a shortcoming is that we can only deal with the case
$\frac{i(\fai)}{q(\fai)}>\max\{p_0,\,\frac{1}{2}\}$, where $p_0\in(0,\,1)$
is described in Proposition \ref{pro est MFFM} below.

(iii) We point out that Theorem \ref{thm2}, Corollary \ref{c1.1}
and Theorem \ref{thm3} have variants as in Theorem \ref{thm1},
the details being omitted.
\end{remark}

The organization of this article is as follows.

In Section \ref{s2}, we give out the proof of Theorem \ref{thm1}.
To this end, we establish some necessary and auxiliary results.  More precisely,
in Subsection \ref{s21}, we establish
the radial maximal function and the
Poisson integral characterizations of $H_\fai(\rn)$ (see Propositions
\ref{pro radial mc} and \ref{pro PIC}  below).

In Subsection \ref{s22}, we introduce a Musielak-Orlicz-Hardy space $H_\fai(\rr^{n+1}_+)$
of harmonic functions (see Definition \ref{def MOH space} below) and show that
the subspace $H_{\fai,\,2}(\rr^{n+1}_+)$ of $H_\fai(\rr^{n+1}_+)$
is isomorphic to $H_\fai(\rn)$ (see Proposition \ref{pro HFC}  below).

In Subsection \ref{s23}, we introduce a Musielak-Orlicz-Hardy space
$\mathcal{H}_\fai(\rr^{n+1}_+)$ of harmonic vectors which satisfy
the generalized Cauchy-Riemann equation \eqref{eqn GCR equation}
(see Definition \ref{def MOHVH space}  below), then we show that
the elements in $\mathcal{H}_\fai(\rr^{n+1}_+)$ have harmonic majorant
and boundary value on $\rn$ (see Lemmas \ref{lem harmonic majorant of MOHVH}
and \ref{lem boudnary value of MOHVH} below).
Moreover, by establishing relations among $\mathcal{H}_\fai(\rr^{n+1}_+)$,
$H_\fai(\rr^{n+1}_+)$ and $H_\fai(\rn)$ (see Propositions \ref{pro HVC} and
\ref{pro MOHtoMOHVH} below),
we obtain the isomorphisms among the spaces $H_\fai(\rn)$,
$H_{\fai,\,2}(\rr^{n+1}_+)$ and $\mathcal{H}_{\fai,\,2}(\rr^{n+1}_+)$
(see Theorem \ref{cor equiv 3spaces} below).

In Subsection \ref{s24}, we prove Theorem \ref{thm1} via Theorem \ref{cor equiv 3spaces}. Furthermore, we also
need the boundedness of Riesz transforms on $H_\fai(\rn)$ (see Corollary \ref{cor bd of Riesz}
 below), which is proved by establishing an interpolation of operators on weighted Hardy spaces
(see Proposition \ref{pro interpolation} below).
We point out that this interpolation result may be
of independent interest, since, by which, we can obtain the boundedness of
many important operators from harmonic analysis and partial differential equations
on $H_\fai(\rn)$; see Corollary \ref{cor bd of CZ} for
the case of Calder\'on-Zygmund operators.

In Section \ref{s3}, we first introduce a Musielak-Orlicz-Hardy space
$\mathcal{H}_{\fai,\,m}(\rr^{n+1}_+)$
of tensor-valued functions (see Definition \ref{def MOTVFH} below),
which plays the same role as $\mathcal{H}_\fai(\rr^{n+1}_+)$
in the first order Riesz transform characterizations (see Remark \ref{r1.2}(ii) below).
Then we prove Theorem \ref{thm2} by a way
similar to that used in the proof of Theorem \ref{thm1}. Finally, by using
an estimate of Uchiyama \cite{Uc84}, we prove Theorem \ref{thm3}.

We end this section by making some conventions on notation.
Throughout the whole article, we always set $\nn:=\{1,2,\ldots\}$
and $\zz_+:=\nn\cup\{0\}$. The {\it differential operator}
$\frac{\pat^{|\az|}}{\pat x_1^{\az_1}\cdots \pat x_n^{\az_n}}$
is denoted simply by $\pat^\az$,
where $\az:=(\az_1,\ldots,\az_n)\in\zz_+^n$ and
$|\az|:=\az_1+\cdots+\az_n$.
Let $C_c^\fz(\rn)$ be the \emph{set of smooth functions with compact
support}. We use $C$ to denote a {\it positive
constant} that is independent of the main parameters involved but
whose value may differ from line to line. We use $C_{(\az,\,\bz,\,\ldots)}$
to denote a positive constant depending on the parameters $\az$,
$\bz$.... If $f\le Cg$, we then write $f\ls
g$ and, if $f\ls g\ls f$, we then write $f\sim g$. For all
$x\in\rr^n$ and $r\in(0,\fz),$ let
$B(x,r):=\{y\in\rr^n:\ |x-y|<r\}.$
Also, for any set $E\subset \rn$, we use $E^\complement$ to denote
$\rn\setminus E$ and $\chi_E$ its {\it characteristic function}, respectively.
For any $s\in \rr$, we let $\lfloor s\rfloor$ to denote the \emph{maximal integer
not more than} $s$. Finally, for $q\in[1,\fz]$, $q':=\frac{q}{q-1}$
denotes the \emph{conjugate exponent} of $q$.

\section{First order Riesz transform characterizations\label{s2}}

\hskip\parindent In this section, we give a complete proof of Theorem \ref{thm1}.
In order to achieve this goal, we need to introduce Musielak-Orlicz-Hardy
type spaces $H_{\fai}(\rr^{n+1}_+)$ of harmonic functions and
$\mathcal{H}_{\fai}(\rr^{n+1}_+)$ of harmonic vectors on the upper half space
$\rr^{n+1}_+$, and establish their relations with $H_{\fai}(\rn)$.
The first three subsections of this section are devoted to the study of these relations.
After this, we prove Theorem \ref{thm1} in Subsection \ref{s24}.

\subsection{Radial maximal function and Poisson integral
characterizations of $H_{\fai}(\rn)$\label{s21}}

\hskip\parindent
Let $p\in[1,\,\fz)$ and $w\in A_p(\rn)$.
It is well known that there exist positive constants $\dz\in(0,\,1)$ and
$C$ such that, for all balls $B_1,\,B_2\subset\rn$ with $B_1\subset B_2$,
\begin{eqnarray}\label{eqn db property}
\frac{w(B_2)}{w(B_1)}\le C \lf(\frac{|B_2|}{|B_1|}\r)^{p}
\end{eqnarray}
and
\begin{eqnarray}\label{eqn rdb property}
\frac{w(B_1)}{w(B_2)}\le C \lf(\frac{|B_1|}{|B_2|}\r)^{\dz}
\end{eqnarray}
(see, for example, \cite{GR85} for more details on the above two inequalities and
other properties of Muckenhoupt weights).

Now, let $\phi\in \mathcal{S}(\rn)$ satisfy
\begin{eqnarray}\label{eqn integral=1}
\int_\rn \phi(x)\,dx=1.
\end{eqnarray}
For any distribution $f\in \mathcal{S}'(\rn)$,
its \emph{radial and non-tangential maximal functions} $\mathcal{M}_{\phi}(f)$
and  $\mathcal{M}^*_{\phi}(f)$ are, respectively, defined by setting, for all $x\in\rn$,
\begin{eqnarray}\label{eqn radial mf}
\mathcal{M}_{\phi}(f)(x):=\dsup_{t\in(0,\,\fz)}\lf|\lf(f*\phi_t\r)(x)\r|
\end{eqnarray}
and
\begin{eqnarray}\label{eqn nontangential mf}
\mathcal{M}_{\phi}^*(f)(x):=\dsup_{|y-x|<t,\,t\in(0,\,\fz)}\lf|\lf(f*\phi_t\r)(y)\r|.
\end{eqnarray}

Liang, Huang and Yang \cite[Theorem 3.7]{L-H-Y12} established the following
non-tangential maximal function characterization of $H_\fai(\rn)$.

\begin{proposition}[\cite{L-H-Y12}]\label{pro nontangential mc}
Let $\fai$ and $\phi\in\mathcal{S}(\rn)$ satisfy, respectively,
Assumption $(\fai)$ and \eqref{eqn integral=1}.
Then $f\in H_\fai(\rn)$ if and only if
$f\in \mathcal{S}'(\rn)$ and
$\mathcal{M}_{\phi}^*(f)\in L^\fai(\rn)$. Moreover, there exists a positive constant
$C$ such that, for all $f\in H_\fai(\rn)$,
\begin{eqnarray*}
\frac{1}{C}\lf\|f\r\|_{H_\fai(\rn)}\le \lf\|\mathcal{M}_{\phi}^*(f)\r\|_{L^\fai(\rn)}\le
C\lf\|f\r\|_{H_\fai(\rn)}.
\end{eqnarray*}
\end{proposition}

The following result provides the radial maximal function characterization of
$H_\fai(\rn)$.

\begin{proposition}\label{pro radial mc}
Let $\fai$ and $\phi\in\mathcal{S}(\rn)$ satisfy, respectively,
Assumption $(\fai)$ and \eqref{eqn integral=1}.
Then $f\in H_\fai(\rn)$ if and only if $f\in \mathcal{S}'(\rn)$ and
$\mathcal{M}_{\phi}(f)\in L^\fai(\rn)$. Moreover, there exists a positive constant
$C$ such that, for all $f\in H_\fai(\rn)$,
\begin{eqnarray*}
\frac{1}{C}\lf\|f\r\|_{H_\fai(\rn)}\le \lf\|\mathcal{M}_{\phi}(f)\r\|_{L^\fai(\rn)}\le
C\lf\|f\r\|_{H_\fai(\rn)}.
\end{eqnarray*}
\end{proposition}

To prove Proposition \ref{pro radial mc}, we need the following boundedness of
the Hardy-Littlewood maximal function on $L^\fai(\rn)$ from
\cite[Corollary 2.8]{L-H-Y12}. Recall that, for all $x\in\rn$, the
\emph{Hardy-Littlewood maximal function} $\mathcal{M}(f)$ of a locally integrable function
$f$ on $\rn$ is defined by setting,
\begin{eqnarray}\label{HL maxiamal function}
\mathcal{M}(f)(x):=\sup_{B\ni x}\frac{1}{|B|}\dint_{B}\lf|f(y)\r|\,dy,
\end{eqnarray}
where the supremum is taken over all balls $B$ in $\rn$ containing $x$.

\begin{lemma}[\cite{L-H-Y12}]\label{lem bd HLMF}
Let $\fai$ satisfy Assumption $(\fai)$ with
the lower type exponent $p\in(1,\,\fz)$ and
$q(\fai)<i(\fai)$, where $q(\fai)$ and $i(\fai)$ are as in
\eqref{1.7} and \eqref{1.6}, respectively.
Then $\mathcal{M}$ is bounded on $L^\fai(\rn)$. Moreover, there exists a positive
constant $C$ such that, for all $f\in L^\fai(\rn)$,
\begin{eqnarray*}
\dint_{\rn}\fai\lf(x,\,\mathcal{M}(f)(x)\r)\,dx\le C \dint_{\rn}\fai\lf(x,\,|f(x)|\r)\,dx.
\end{eqnarray*}
\end{lemma}

We now turn to the proof of Proposition \ref{pro radial mc}.

\begin{proof}[Proof of Proposition \ref{pro radial mc}]
The direction that $f\in H_\fai(\rn)$ implies $\mathcal{M}_{\phi}(f)\in L^\fai(\rn)$
is an easy consequence of Proposition \ref{pro nontangential mc} and the fact that,
for all $x\in\rn$, $\mathcal{M}_{\phi}(f)(x)\le\mathcal{M}_{\phi}^*(f)(x)$,
the details being omitted.

Now, let $f\in \mathcal{S}'(\rn)$ satisfy $\mathcal{M}_{\phi}(f)\in L^\fai(\rn)$.
Based on Proposition \ref{pro nontangential mc},
we prove another direction of Proposition \ref{pro radial mc} by showing that
\begin{eqnarray}\label{eqn comparison of RM NM}
\lf\|\mathcal{M}_{\phi}^*(f)\r\|_{L^\fai(\rn)}\ls
\lf\|\mathcal{M}_{\phi}(f)\r\|_{L^\fai(\rn)}.
\end{eqnarray}
Indeed, for any $\epsilon\in(0,\,1)$, $N\in\nn$ sufficiently large and $x\in\rn$, let
\begin{eqnarray*}
\mathcal{M}_{\phi,\,\epsilon,\,N}^*(f)(x):=\dsup_{|x-y|<t<\frac1\epsilon}
\lf|\lf(f*\phi_t\r)(y)\r|\lf(\frac{t}{t+\epsilon}\r)^N\lf(1+\epsilon|y|\r)^{-N}.
\end{eqnarray*}
It is easy to see that, for all $x\in\rn$, $\lim_{\epsilon\to 0^+,\,N\to \fz} \mathcal{M}_{\phi,\,\epsilon,\,N}^*(f)(x)=\mathcal{M}_{\phi}^*(f)(x)$.

We first claim that, for all $\lz\in(0,\,\fz)$, there exists a positive constant
$C_{(N,\,n,\,\fai,\,\phi)}$, depending only on $N$, $n$, $\fai$ and $\phi$, such that
\begin{eqnarray}\label{eqn est1 pro RMC}
\dint_{\rn}\fai\lf(x,\,\frac{\mathcal{M}_{\phi,\,\epsilon,\,N}^*(f)(x)}{\lz}\r)\,dx\le C_{(N,\,n,\,\fai,\,\phi)}
\dint_{\rn}\fai\lf(x,\,\frac{\mathcal{M}_\phi(f)(x)}{\lz}\r)\,dx.
\end{eqnarray}

To prove this claim, for all $x\in\rn$, let
\begin{eqnarray*}
\wz{\mathcal{M}}_{\phi,\,\epsilon,\,N}^*(f)(x):=\dsup_{|x-y|<t<\frac1\epsilon}
t\lf|\nabla_y\lf(f*\phi_t\r)(y)\r|\lf(\frac{t}{t+\epsilon}\r)^N\lf(1+\epsilon|y|\r)^{-N}.
\end{eqnarray*}
From the proof of \cite[(6.4.22)]{Gr09},
we deduce that, for any $p\in(0,\,\fz)$, $\epsilon\in(0,\,1)$ and $N\in\nn$,
there exists a positive constant
$C_{(N,\,n,\,\fai,\,\phi)}$ such that, for all $x\in\rn$,
\begin{eqnarray}\label{eqn est2 pro RMC}
\wz{\mathcal{M}}_{\phi,\,\epsilon,\,N}^*(f)(x)\le C_{(N,\,n,\,\fai,\,\phi)}
\lf\{\mathcal{M} \lf(\lf[\mathcal{M}_{\phi,\, \epsilon,\,N}^*(f)\r]^p\r)(x)\r\}^{1/p},
\end{eqnarray}
where $\mathcal{M}$ denotes the Hardy-Littlewood maximal function as in
\eqref{HL maxiamal function}.

Now, let
\begin{eqnarray*}
E_{\epsilon,\,N}:=\lf\{x\in\rn:\ \wz{\mathcal{M}}_{\phi,\,\epsilon,\,N}^*(f)(x)\le
C_0 \mathcal{M}_{\phi,\,\epsilon,\,N}^*(f)(x)\r\},
\end{eqnarray*}
where $C_0$ is a sufficiently large constant whose size will be determined later.
For all $(x,\,t)\in\rr^{n+1}_+$, let $\fai_p(x,\,t):=\fai(x,\,t^{1/p})$.
By the definition of $i(\fai)$, we know that there exists $p_0\in(0,\,i(\fai))$ such that,
for any $x\in\rn$, $\fai(x,\,\cdot)$ is of lower type ${p}_0$.
It is easy to see that
$i(\fai_p)=\frac{i(\fai)}{p}$ and, for any $x\in\rn$, $\fai_p(x,\cdot)$ is of lower type $\frac{p_0}p$.
Thus, by taking $p$ sufficiently small, we obtain $q(\fai_p)<i(\fai_p)$, which, together with \eqref{eqn est2 pro RMC}, Lemma \ref{lem bd HLMF} and
the lower type $p_0$ property of $\fai(x,\cdot)$,
implies that there exists a positive constant $C_{(\fai)}$ satisfying that,
for any $\lz\in(0,\,\fz)$,
\begin{eqnarray}\label{eqn est4 pro RMC}
&&\int_{(E_{\epsilon,\,N})^\complement}\fai
\lf(x,\,\frac{\mathcal{M}_{\phi,\,\epsilon,\,N}^*(f)(x)}{\lz}\r)\,dx\\
&&\nonumber\hs\le C_{(\fai)}\lf(\frac{1}{C_0}\r)^{{p}_0} \int_{(E_{\epsilon,\,N})^\complement}\fai
\lf(x,\,\frac{\wz{\mathcal{M}}_{\phi,\,\epsilon,\,N}^*(f)(x)}{\lz}\r)\,dx\\
&&\nonumber\hs\le C_{(N,\,n,\,\fai,\,\phi)} \lf(\frac{1}{C_0}\r)^{{p}_0} \int_{(E_{\epsilon,\,N})^\complement}\fai_p
\lf(x,\,\frac{\mathcal{M}([\mathcal{M}_{\phi,\, \epsilon,\,N }^*(f)]^p)(x)}
{\lz^p}\r)\,dx\\&&\nonumber\hs\le C_{(N,\,n,\,\fai,\,\phi)} \lf(\frac{1}{C_0}\r)^{{p}_0} \int_{\rn}\fai
\lf(x,\,\frac{\mathcal{M}_{\phi,\, \epsilon,\,N }^*(f)(x)}
{\lz}\r)\,dx.
\end{eqnarray}
By taking $C_0$ in \eqref{eqn est4 pro RMC} sufficiently large so that $C_{(N,\,n,\,\fai,\,\phi)} (\frac{1}{C_0})^{{p}_0}<\frac{1}{2}$,
we see that
\begin{eqnarray}\label{eqn est3 pro RMC}
\dint_{\rn}\fai
\lf(x,\,\frac{\mathcal{M}_{\phi,\,\epsilon,\,N}^*(f)(x)}{\lz}\r)\,dx\le 2
\dint_{E_{\epsilon,\,N}}\fai
\lf(x,\,\frac{\mathcal{M}_{\phi,\,\epsilon,\,N}^*(f)(x)}{\lz}\r)\,dx.
\end{eqnarray}
Moreover, from \cite[(6.4.27)]{Gr09},
it follows that, for all $r<i(\fai)$
and $x\in E_{\epsilon,\,N}$,
\begin{eqnarray*}
\mathcal{M}_{\phi,\,\epsilon,\,N}^*(f)(x)\le C_{(N,\,n,\,\fai,\,\phi)}
\lf\{\mathcal{M}\lf(\lf[\mathcal{M}_{\phi}(f)\r]^r\r)(x)\r\}^{1/r},
\end{eqnarray*}
which, together with \eqref{eqn est3 pro RMC} and an argument similar
to that used in the estimate \eqref{eqn est4 pro RMC}, implies that
\eqref{eqn est1 pro RMC} holds true.

Now, we finish the proof of Proposition \ref{pro radial mc} by using the above claim.
Observe that, for $x\in\rn$,
\begin{eqnarray*}
\mathcal{M}_{\phi,\,\epsilon,\,N}^*(f)(x)\ge \frac{2^{-N}}{(1+\epsilon|x|)^N}\dsup_{|x-y|<t<\frac1\epsilon}
\lf|\lf(f*\phi_t\r)(y)\r|\lf(\frac{t}{t+\epsilon}\r)^N=:F_{\epsilon,\,N}(x).
\end{eqnarray*}
It is easy to see, for each $N$ and $x$, $F_{\epsilon,\,N}(x)$ is increasing to
$2^{-N}\mathcal{M}_{\phi}^*(f)(x)$ as $\epsilon\to 0^+$, which, combined with
\eqref{eqn est1 pro RMC} and Lebesgue's monotone convergence theorem,
implies that
\begin{eqnarray*}
\dint_{\rn}\fai\lf(x,\,\frac{\mathcal{M}_{\phi}^*(f)(x)}{\lz}\r)\,dx\le
C_{(N,\,n,\,\fai,\,\phi)} \dint_{\rn}\fai
\lf(x,\,\frac{\mathcal{M}_\phi(f)(x)}{\lz}\r)\,dx.
\end{eqnarray*}

In particular, $\mathcal{M}_\phi(f)\in L^\fai(\rn)$ implies that
$\mathcal{M}_{\phi}^*(f)\in L^\fai(\rn)$. This, together with a repetition
of the above argument used in the  proof of  the estimate \eqref{eqn est1 pro RMC} with
$\epsilon:=0$ and $N:=\fz$ in $\mathcal{M}_{\phi,\,\epsilon,\,N}^*(f)$ and $\wz{\mathcal{M}}_{\phi,\,\epsilon,\,N}^*(f)$, implies that
\begin{eqnarray*}
\dint_{\rn}\fai\lf(x,\,\frac{\mathcal{M}_{\phi}^*(f)(x)}{\lz}\r)\,dx\le C_{(n,\,\fai,\,\phi)}\dint_{\rn}\fai\lf(x,\,\frac{\mathcal{M}_\phi(f)(x)}{\lz}\r)\,dx.
\end{eqnarray*}
This finishes the proof of Proposition \ref{pro radial mc}.
\end{proof}

We also need the following Poisson integral characterization of $H_\fai(\rn)$.
Recall that a distribution $f\in \mathcal{S}'(\rn)$ is called a \emph{bounded distribution},
if, for any $\phi\in \mathcal{S}(\rn)$, $f*\phi\in L^\fz(\rn)$. For all $(x,\,t)\in\rr^{n+1}_+$, let
\begin{eqnarray}\label{eqn PK}
P_t(x):=C_{(n)}\frac{t}{(t^2+|x|^2)^{(n+1)/2}}
\end{eqnarray}
be the \emph{Poisson kernel}, where $C_{(n)}$ is the same as in
\eqref{eqn RZ}. It is well known that, if $f$ is a
bounded distribution, then $f*P_t$ is a well-defined, bounded and smooth function.
Moreover, $f*P_t$ is harmonic on $\rr^{n+1}_+$ (see \cite[p.\,90]{St91}).

Recall that, in \cite[p.\,91, Theorem 1]{St91}, Stein established the
Poisson integral characterization of the classical Hardy space
$H^p(\rn)$ by using some pointwise estimates.
These estimates can directly be used in our setting to obtain
the following proposition, the details being omitted.

\begin{proposition}\label{pro PIC}
Let $\fai$ satisfy Assumption $(\fai)$
and $f\in \mathcal{S}'(\rn)$ be a bounded distribution.
Then $f\in H_\fai(\rn)$ if and only if
$f_{P}^*\in L^\fai(\rn)$, where, for all $x\in\rn$,
\begin{eqnarray*}
f_{P}^*(x):=\dsup_{|y-x|<t,\,t\in(0,\,\fz)}\lf|(f*P_t)(y)\r|.
\end{eqnarray*}
Moreover,  there exists a positive constant $C$ such that, for all $f\in H_\fai(\rn)$,
\begin{eqnarray*}
\frac{1}{C}\lf\|f\r\|_{H_\fai(\rn)}\le \lf\|f_{P}^*\r\|_{L^\fai(\rn)}\le C
\lf\|f\r\|_{H_\fai(\rn)}.
\end{eqnarray*}
\end{proposition}
\begin{remark}\label{rem PIC}
We point out that the statement of Proposition \ref{pro PIC}
is a little bit different from that of \cite[p.\,91, Theorem 1]{St91} in that
here we  assume, a priori, that $f\in \mathcal{S}'(\rn)$ is a bounded distribution.
This is because that, for an arbitrary $f\in H_\fai(\rn)$, we cannot show that $f$
is a bounded distribution without  any additional assumptions on $f$ or $\fai$.
However,  by the facts that the set
$H_\fai(\rn)\cap L^2(\rn)$ is dense in $H_\fai(\rn)$ and
there exist $\psi_1$, $\psi_2\in \mathcal{S}(\rn)$ and $h\in L^1(\rn)$ such that,
for all $t\in(0,\,\fz)$,
\begin{eqnarray}\label{eqn decomposition PK}
P_t=(\psi_1)_t*h_t+(\psi_2)_t
\end{eqnarray}
(see \cite[p.\,90]{St91}), we know that, for every $f\in H_\fai(\rn)$,
we can define $f*P_t$ by setting, for all $x\in\rn$ and $t\in(0,\,\fz)$,
$(f*P_t)(x):=\lim_{k\to \fz} (f_k*P_t)(x),$
where $\{f_k\}_{k\in\nn}\subset (H_\fai(\rn)\cap L^2(\rn))$
satisfies $\lim_{k\to \fz} f_k=f$ in $H_\fai(\rn)$ and hence in $\mathcal{S}'(\rn)$.
\end{remark}

\subsection{Musielak-Orlicz-Hardy spaces $H_\fai(\rr^{n+1}_+)$ of harmonic functions\label{s22}}
\hskip\parindent
In this subsection, we introduce the Musielak-Orlicz-Hardy space
$H_\fai(\rr^{n+1}_+)$ of harmonic functions and establish its
relation with $H_\fai(\rn)$.

To this end, let $u$ be a function on $\rr^{n+1}_+$.
Its \emph{non-tangential maximal function} $u^*$ is defined by setting,
for all $x\in\rn$,
\begin{eqnarray*}
u^*(x):=\dsup_{|y-x|<t,\,t\in(0,\,\fz)}|u(y,\,t)|.
\end{eqnarray*}

Recall that a function $u$ on $\rr^{n+1}_+$ is said to be \emph{harmonic} if
$(\bdz_x+\pat^2_t)u(x,\,t)=0$ for all $(x,\,t)\in \rr^{n+1}_+$.

\begin{definition}\label{def MOHH space}
Let $\fai$ satisfy Assumption $(\fai)$.
The \emph{Musielak-Orlicz-Hardy space of harmonic functions}, $H_\fai(\rr^{n+1}_+)$,
is defined to be the space
of all harmonic functions $u$ on $\rr^{n+1}_+$ such that $u^*\in L^\fai(\rn)$.
Moreover, for all $u\in H_\fai(\rr^{n+1}_+)$, its \emph{quasi-norm}
is defined by $\|u\|_{H_\fai(\rr^{n+1}_+)}:= \|u^*\|_{L^\fai(\rn)}$.
\end{definition}

Recall also the following notion of the Hardy space $H^p(\rr^{n+1}_+)$ of harmonic
functions with $p\in(1,\,\fz)$  from \cite{SW71} (see also \cite{ABR01}).

\begin{definition}[\cite{SW71}]\label{def MOHH space p>1}
Let $p\in(1,\,\fz)$.
The \emph{Hardy space $H^p(\rr^{n+1}_+)$ of harmonic functions}
is defined to be the space of all harmonic functions $u$ on $\rr^{n+1}_+$
such that, for all $t\in(0,\,\fz)$, $u(\cdot,\,t)\in L^p(\rn)$.
Moreover, for all $u\in H^p(\rr^{n+1}_+)$, its \emph{norm}
is defined by
$$\|u\|_{H^p(\rr^{n+1}_+)}:= \dsup_{t\in(0,\,\fz)}\|u(\cdot,\,t)\|_{L^p(\rn)}.$$
For $\fai$ as in Definition \ref{def MOHH space}, let
\begin{eqnarray*}
H_{\fai,\,2}(\rr^{n+1}_+):=\overline{H_\fai(\rr^{n+1}_+)\cap H^2(\rr^{n+1}_+)}^{\|\cdot\|_{H_\fai(\rr^{n+1}_+)}}
\end{eqnarray*}
be the \emph{completion of the set} $H_\fai(\rr^{n+1}_+)\cap H^2(\rr^{n+1}_+)$
\emph{under the quasi-norm} $\|\cdot\|_{H_\fai(\rr^{n+1}_+)}$.
\end{definition}

\begin{remark}\label{rem MOHHFS P>1}
For any $u\in H_{\fai}(\rr^{n+1}_+)\cap
H^2(\rr^{n+1}_+)$, from the Poisson integral
characterization of $H^2(\rr^{n+1}_+)$
(see, for example, \cite[Theorem 7.17]{ABR01}),
we deduce that $u$ satisfies the
following \emph{semigroup formula} that,
for all $x\in\rn$ and $s,\,t\in(0,\,\fz)$,
\begin{eqnarray}\label{eqn SGF}
u(x,\,s+t)=(u(\cdot,\,s)*P_t)(x),
\end{eqnarray}
where $P_t$ denotes the Poisson kernel as in \eqref{eqn PK}.
This formula was first introduced by Bui in \cite{B90}. Moreover,
let $p\in(0,\,1]$, $w\in A_\fz(\rn)$ and  $H_w^p(\rr^{n+1}_+)$ be
the weighted Hardy space of harmonic functions defined as in Definition
\ref{def MOHH space} via the radial maximal functions.
Let $\overline H_w^p(\rr^{n+1}_+)$ be the \emph{closure} in $H_w^p(\rr^{n+1}_+)$
of the subspace of those functions in $H_w^p(\rr^{n+1}_+)$ for which the
semigroup formula \eqref{eqn SGF} holds. Bui proposed the question
that, under what condition, $\overline H_w^p(\rr^{n+1}_+)$
is equivalent to $H_w^p(\rr^{n+1}_+)$. It is known that if, for some $d\in(0,\,\fz)$,
$w$ satisfies the following extra \emph{condition}
that, for all $x\in\rn$ and
$\rho\in(0,\,1)$,
\begin{eqnarray*}
\dint_{B(x,\,\rho)}w(y)\,dy\gs \rho^d,
\end{eqnarray*}
then $\overline H_w^p(\rr^{n+1}_+)=H_w^p(\rr^{n+1}_+)$ (see \cite[Remark 3.3]{B90}).
In particular, if $w\equiv1$,  the above two spaces coincide.
We refer the reader to \cite[Remark 3.3]{B90} for more details.
\end{remark}

The following proposition shows that the spaces,
$H_{\fai,\,2}(\rr^{n+1}_+)$
and $H_\fai(\rn)$, are isomorphic to each other via the Poisson integral.

\begin{proposition}\label{pro HFC}
Let $\fai$ satisfy Assumption $(\fai)$ and $u$ be
a harmonic function on $\rr^{n+1}_+$. Then $u\in H_{\fai,\,2}(\rr^{n+1}_+)$ if and only
if there exists $f\in H_\fai(\rn)$ such that, for all $(x,\,t)\in\rr^{n+1}_+$,
$u(x,\,t)=\lf(f*P_t\r)(x),$
where $(f*P_t)(x)$ is defined as in Remark \ref{rem PIC}.
Moreover, there exists a positive constant $C$, independent of $f$ and $u$,
such that
\begin{eqnarray*}
\frac{1}{C}\|f\|_{H_\fai(\rn)}\le \|u\|_{H_\fai(\rr^{n+1}_+)}\le C \|f\|_{H_\fai(\rn)}.
\end{eqnarray*}
\end{proposition}

\begin{proof}
By Definition \ref{def MOHH space p>1} and Remark \ref{rem defMOH}(iii),
to prove Proposition \ref{pro HFC},
it suffices to show that the Poisson integral $P_t$ is an isomorphism
from $(H_\fai(\rn)\cap L^2(\rn),\,\|\cdot\|_{H_\fai(\rn)})$ to $(H_\fai(\rr^{n+1}_+)\cap H^2(\rr^{n+1}_+),\,\|\cdot\|_{H_\fai(\rr^{n+1}_+)})$.
Recall that the Poisson integral is an isomorphism from $L^2(\rn)$ to
$ H^2(\rr^{n+1}_+)$ (see, for example, \cite[Theorem 7.17]{ABR01}).

The inclusion that $$P_t\lf(H_\fai(\rn)\cap L^2(\rn),\,\|\cdot\|_{H_\fai(\rn)}\r)
\subset \lf(H_\fai(\rr^{n+1}_+)\cap H^2(\rr^{n+1}_+),\,\|\cdot\|_{H_\fai(\rr^{n+1}_+)}\r)$$
is an easy consequence of Proposition \ref{pro PIC}, the details being omitted.

We now turn to the inverse inclusion.
Let $u\in H_\fai(\rr^{n+1}_+)\cap H^2(\rr^{n+1}_+)$.
For any $x\in\rn$ and $\epsilon\in(0,\,\fz)$, let $u_\epsilon(x,\,t):=u(x,\,t+\epsilon)$.
Since $u\in H^2(\rr^{n+1}_+)$, we know that $u_\epsilon$ can be represented
as a Poisson integral: $u_\epsilon(x,\,t)=\lf(f_\epsilon*P_t\r)(x)$
for all $(x,\,t)\in\rr^{n+1}_+$, where $f_\epsilon(x):=u(x,\,\epsilon)$.
Moreover, from Proposition \ref{pro PIC} and the definition of the
non-tangential maximal function, it follows that
\begin{eqnarray}\label{2.y1}
\qquad\quad\dsup_{\epsilon\in(0,\,\fz)}\lf\|f_\epsilon\r\|_{H_\fai(\rn)}\sim
\dsup_{\epsilon\in(0,\,\fz)}
\lf\|\lf(f_\epsilon*P_t\r)^*\r\|_{L^\fai(\rn)} \sim \dsup_{\epsilon\in(0,\,\fz)}
\lf\|u^*_\epsilon\r\|_{L^\fai(\rn)} \ls
\lf\|u^*\r\|_{L^\fai(\rn)}.
\end{eqnarray}
Thus, $\{f_\epsilon\}_{\epsilon\in(0,\,\fz)}$ is a bounded set in $H_\fai(\rn)$ and
hence in $\mathcal{S}'(\rn)$ (see \cite[Proposition 5.1]{Ky11}).
By the weak compactness
of $\mathcal{S}'(\rn)$ (see, for example, \cite[p.\,119]{St91}), we conclude that there exist an
$f\in \mathcal{S}'(\rn)$ and a subsequence $\{f_{k}\}_{k\in\nn}$ such that $\{f_{k}\}_{k\in\nn}$
converges weakly to $f$ in $\mathcal{S}'(\rn)$.
This, together with \eqref{eqn decomposition PK}, implies that,
for all $(x,\,t)\in\rr^{n+1}_+$,
$$\lim_{k\to \fz} (f_k*P_t)(x)=(f*P_t)(x)=u(x,\,t).$$
Thus, by Proposition \ref{pro PIC},  Fatou's lemma and \eqref{2.y1},
we conclude that
\begin{eqnarray*}
\lf\|f\r\|_{H_\fai(\rn)}&&\sim\lf\|\lim_{k\to \fz}  (f_k)_{P}^*\r\|_{L^\fai(\rn)}\ls
\mathop{\underline{\lim}}_{k\to \fz} \lf\| (f_k)_{P}^*\r\|_{L^\fai(\rn)}\\
&&\sim \mathop{\underline{\lim}}_{k\to \fz} \|f_k\|_{H_\fai(\rn)}\ls
\|u^*\|_{L^\fai(\rn)}\sim \|u\|_{H_\fai(\rr^{n+1}_+)},
\end{eqnarray*}
which immediately implies that $f\in H_\fai(\rn)$, $u(x,\,t)=f*P_t(x)$ and hence completes
the proof of Proposition \ref{pro HFC}.
\end{proof}

\subsection{Musielak-Orlicz-Hardy spaces $\mathcal{H}_\fai(\rr^{n+1}_+)$ of harmonic vectors\label{s23}}
\hskip\parindent
In this subsection, we study the Musielak-Orlicz-Hardy space
$\mathcal{H}_\fai(\rr^{n+1}_+)$
consisting of vectors of harmonic functions
which satisfy the so-called
generalized Cauchy-Riemann equation.
To be precise, let $F:=\{u_0,\,u_1,\,\ldots,\,u_n\}$
be a harmonic vector on $\rr^{n+1}_+$. Then $F$
is said to satisfy the \emph{generalized Cauchy-Riemann equation},
if, for all $j,\,k\in\{0,\,\ldots,\,n\}$,
\begin{eqnarray}\label{eqn GCR equation}
\begin{cases}
\dsum_{j=0}^n \frac{\pat u_j}{\pat x_j}=0,\\
{\dfrac{\pat u_j}{\pat x_k}=\frac{\pat u_k}{\pat x_j}},
\end{cases}
\end{eqnarray}
where, for $(x,\,t)\in\rr^{n+1}_+$, we let $x:=(x_1,\,\ldots,\,x_n)$ and $x_0:=t$.

\begin{definition}\label{def MOHVH space}
Let $\fai$ be a Musielak-Orlicz function satisfying Assumption $(\fai)$.
The \emph{Musielak-Orlicz-Hardy space} $\mathcal{H}_\fai(\rr^{n+1}_+)$
\emph{of harmonic vectors} is defined to be the space of all
harmonic vectors $F:=\{u_0,\,u_1,\,\ldots,\,
u_n\}$ on $\rr^{n+1}_+$ satisfying \eqref{eqn GCR equation}
such that, for all $t\in(0,\,\fz)$,
$$|F(\cdot,\,t)|:=\lf\{\sum_{j=0}^n
|u_j(\cdot,\,t)|^2\r\}^{1/2}\in L^\fai(\rn).$$
Moreover, for any $F\in \mathcal{H}_\fai(\rr^{n+1}_+)$,
its \emph{quasi-norm} is defined by setting,
\begin{eqnarray*}
\lf\|F\r\|_{\mathcal{H}_\fai(\rr^{n+1}_+)}:=\dsup_{t\in(0,\,\fz)}
\lf\|\lf|F(\cdot,\,t)\r|\r\|_{L^\fai(\rn)}.
\end{eqnarray*}

For $p\in(1,\,\fz)$, the \emph{Musielak-Orlicz-Hardy space} $\mathcal{H}^p(\rr^{n+1}_+)$
\emph{of harmonic vectors} is defined as $ \mathcal{H}_\fai(\rr^{n+1}_+)$
with $L^\fai(\rn)$ replaced by $L^p(\rn)$. In particular, for
any $F\in \mathcal{H}^p(\rr^{n+1}_+)$, its \emph{norm} is defined by setting,
\begin{eqnarray*}
\lf\|F\r\|_{\mathcal{H}^p(\rr^{n+1}_+)}:=\dsup_{t\in(0,\,\fz)}
\lf\|\lf|F(\cdot,\,t)\r|\r\|_{L^p(\rn)}.
\end{eqnarray*}
Moreover, let
\begin{eqnarray*}
\mathcal{H}_{\fai,\,2}(\rr^{n+1}_+):=\overline{\mathcal{H}_{\fai}(\rr^{n+1}_+)
\cap \mathcal{H}^{2}(\rr^{n+1}_+)}^{\|\cdot\|_{\mathcal{H}_{\fai}(\rr^{n+1}_+)}}
\end{eqnarray*}
be the \emph{completion} of the set $\mathcal{H}_{\fai}(\rr^{n+1}_+)
\cap \mathcal{H}^{2}(\rr^{n+1}_+)$ under the quasi-norm
$\|\cdot\|_{\mathcal{H}_{\fai}(\rr^{n+1}_+)}$.
\end{definition}

\begin{remark}\label{rem defMOHHV}
The space $\mathcal{H}^p(\rr^{n+1}_+)$ was first introduced by Stein and
Weiss to give a higher dimensional generalization of the Hardy space
on the upper plane (see \cite{SW60,SW68,SW71} for more details).
\end{remark}

For any $F\in \mathcal{H}_\fai(\rr^{n+1}_+)$, we have the following technical lemmas,
respectively, on the harmonic majorant  and the boundary value of $F$.

\begin{lemma}\label{lem harmonic majorant of MOHVH}
Assume that the function $\fai$ satisfies Assumption $(\fai)$
with $\frac{i(\fai)}{q(\fai)}>\frac{n-1}{n}$ and $F:=\{u_0,\,u_1,\,\ldots,\,
u_n\}\in \mathcal{H}_\fai(\rr^{n+1}_+)$, where $i(\fai)$ and $q(\fai)$
are as in \eqref{1.6} and \eqref{1.7}, respectively.
Then, for all $q\in[\frac{n-1}{n},\,
\frac{i(\fai)}{q(\fai)})$, $a\in(0,\,\fz)$ and $(x,\,t)\in \rr^{n+1}_+$,
\begin{eqnarray}\label{eqn harmoic marjant}
\lf|F(x,\,t+a)\r|^q\le \lf(\lf|F(x,\,a)\r|^q *P_t\r)(x),
\end{eqnarray}
where $P_t$ is the Poisson kernel as in \eqref{eqn PK}.
\end{lemma}

\begin{proof}
For all $t\in[0,\fz)$, let
\begin{eqnarray*}
K(|F|^q,\,t):=\dint_{\rn} \frac{|F(x,\,t)|^q}{(|x|+1+t)^{n+1}}\,dx.
\end{eqnarray*}
Since $|F|^q$ is subharmonic on $\rr^{n+1}_+$ (see \cite[p.\,234, Theorem 4.14]{St70}),
by \cite[p.\,245, Theorem 2]{Nu73}, in order to prove \eqref{eqn harmoic marjant}, it
suffices to show that
\begin{eqnarray}\label{eqn est1 lem HMMOH}
\lim_{t\to \fz}K(|F|^q,\,t)=0.
\end{eqnarray}
We now prove \eqref{eqn est1 lem HMMOH}. Write
\begin{eqnarray}\label{eqn est2 lem HMMOH}
\quad K(|F|^q,\,t)&&=\dint_{\{x\in\rn:\ |F(x,\,t)|\ge 1\}}\frac{|F(x,\,t)|^q}{(|x|+1+t)^{n+1}}\,dx
+\dint_{\{x\in\rn:\ |F(x,\,t)|< 1\}}\cdots\\
&&\nonumber=:\mathrm{I}+\mathrm{II}.
\end{eqnarray}

We first estimate $\mathrm{I}$. By choosing $r\in(q(\fai),\,\fz)$
satisfying $r<\frac{i(\fai)n}{n-1}$
and $\frac{n-1}{n}\le q<\frac{i(\fai)}{r}$, we know that, for all $(x,\,t)\in\rr^{n+1}_+$,
$\fai(\cdot,\,t)\in A_r(\rn)$ and $\fai(x,\,\cdot)$ is of lower type $qr$,
which, together with H\"older's inequality, further implies that
\begin{eqnarray}\label{eqn est3 lem HMMOH}
\mathrm{I}&&\ls \lf\{\dint_{\{x\in\rn:\ |F(x,\,t)|\ge 1\}}|F(x,\,t)|^{qr}
\fai(x,\,1)\,dx\r\}^{\frac{1}{r}}\\
&&\nonumber\hs\times \lf\{\dint_{\{x\in\rn:\ |F(x,\,t)|\ge 1\}}
\frac{1}{(|x|+1+t)^{(n+1)r'}}\lf[\fai(x,\,1)\r]^{-r'/r}\,dx\r\}^{\frac{1}{r'}}\\
&&\nonumber\ls \lf\{\dint_{\{x\in\rn:\ |F(x,\,t)|\ge 1\}}
\fai(x,\,|F(x,\,t)|)\,dx\r\}^{1/r} \\
&&\nonumber\hs\times\lf\{\dint_{\{x\in\rn:\ |F(x,\,t)|\ge 1\}}\frac{1}
{(|x|+1+t)^{(n+1)r'}} {\lf[\fai(x,\,1)\r]^{-r'/r}}\,dx\r\}^{1/r'}.
\end{eqnarray}
Since $\fai(\cdot,\,1)\in A_r(\rn)$, we see
$w(\cdot):=[\fai(\cdot,\,1)]^{-r'/r}\in A_{r'}(\rn)$ (see, for example,
\cite[p.\,394, Theorem 1.14(c)]{GR85}),
which, together with \cite[Lemma 1]{HMW73}, implies that $w$ satisfies the so-called
$B_{r'}(\rn)$-\emph{condition}, namely, for all $x\in\rn$,
\begin{eqnarray}\label{eqn Br condition}
\dint_{\rn}\frac{w(y)}{(t+|x-y|)^{nr'}}\,dy\ls t^{-nr'}\dint_{B(x,\,t)}w(y)\,dy.
\end{eqnarray}
By this, together with \eqref{eqn est3 lem HMMOH}, we further see that
\begin{eqnarray}\label{eqn est4 lem HMMOH}
\mathrm{I}\ls&& \frac{1}{1+t}\lf\{\dint_{\rn}
\fai(x,\,|F(x,\,t)|)\,dx\r\}^{1/r} \lf\{\dint_{\rn} \frac{\lf[\fai(x,\,1)\r]^{-r'/r}}
{(|x|+1+t)^{nr'}}\,dx\r\}^{1/r'}\\
\nonumber\le && C_{(\fai),\,1} \frac{1}{1+t},
\end{eqnarray}
where $C_{(\fai),\,1}$ is a positive constant, depending on $\fai$, but independent of
$t$.

To estimate the term $\mathrm{II}$, let $\wz{r}:=\frac{1}{q}$. It is easy to see that
$r<\frac{i(\fai)}{q}\le \frac{1}{q}=\wz{r}$. Thus, $\fai(\cdot,\,1)\in A_{\wz{r}}(\rn)$,
which, together with H\"older's inequality, the upper type 1 property of $\fai(x,\cdot)$
and \eqref{eqn Br condition}, implies that
\begin{eqnarray}\label{eqn est5 lem HMMOH}
\mathrm{II}&&\ls\lf\{\dint_{\rn}
|F(x,\,t)|^{q\wz r}\fai(x,\,1)\,dx\r\}^{1/\wz{r}}\\
&&\hs\nonumber\times \lf\{\dint_{B(0,\,1)}
\frac{1}{(|x|+1+t)^{(n+1)\wz{r}'}}\lf[\fai(x,\,1)\r]^{-\wz{r}'/\wz{r}}\,dx\r\}^{1/\wz{r}'}\\
&& \nonumber\ls \frac{1}{1+t}\lf\{\dint_{\rn}
\fai(x,\,|F(x,\,t)|)\,dx\r\}^{1/\wz{r}} \lf\{\dint_{B(0,\,1)}
\lf[\fai(x,\,1)\r]^{-\wz{r}'/\wz{r}}\,dx\r\}^{1/\wz{r}'}
\\ \nonumber &&\le C_{(\fai),\,2} \frac{1}{1+t},
\end{eqnarray}
where $C_{(\fai),\,2}$ is a positive constant, depending on $\fai$, but independent of
$t$.
Combining \eqref{eqn est2 lem HMMOH}, \eqref{eqn est4 lem HMMOH} and
\eqref{eqn est5 lem HMMOH}, we see that \eqref{eqn est1 lem HMMOH} holds true.
This finishes the proof of Lemma \ref{lem harmonic majorant of MOHVH}.
\end{proof}

\begin{lemma}\label{lem boudnary value of MOHVH}
Assume that the function $\fai$ satisfies Assumption $(\fai)$ with
$\frac{i(\fai)}{q(\fai)}>\frac{n-1}{n}$ and
$F:=\{u_0,\,u_1,\,\ldots,\,u_n\}\in \mathcal{H}_\fai(\rr^{n+1}_+)$, where
$i(\fai)$ and $q(\fai)$ are as in \eqref{1.6} and
\eqref{1.7}, respectively. Then there
exists $h\in L^\fai(\rn)$ such that $\lim_{t\to 0}|F(\cdot,\,t)|=h(\cdot)$ in $L^\fai(\rn)$
and $h$ is the non-tangential limit of $F$ as $t\to 0$ almost everywhere,
namely, for almost every $x_0\in\rn$,
$\lim_{(x,\,t)\to (x_0,\,0^+)}|F(x,\,t)|=h(x_0)$
for all $(x,\,t)$ in the cone $\bgz(x_0):=\{(x,\,t)\in\rr^{n+1}_+:\ |x-x_0|<t\}$.
Moreover, for all $q\in[\frac{n-1}{n},\,\frac{i(\fai)}{q(\fai)})$ and $(x,\,t)\in\rr^{n+1}_+$,
\begin{eqnarray}\label{eqn harmoic marjant BV}
\lf|F(x,\,t)\r|\le \lf[\lf(h^q*P_t\r)(x)\r]^{1/q},
\end{eqnarray}
where $P_t$ is the Poisson kernel as in \eqref{eqn PK}.
\end{lemma}

\begin{proof}
For $F\in \mathcal{H}_\fai(\rr^{n+1}_+)$ and all $(x,\,t)\in\rr^{n+1}_+$, let
\begin{eqnarray*}
F_1(x,\,t):=\chi_{\{(x,\,t)\in\rr^{n+1}_+:\ |F(x,\,t)|\ge 1\}}F(x,\,t)
\end{eqnarray*}
and
\begin{eqnarray*}
F_2(x,\,t):=\chi_{\{(x,\,t)\in\rr^{n+1}_+:\ |F(x,\,t)|< 1\}}F(x,\,t).
\end{eqnarray*}

Let $r\in(q(\fai),\,\fz)$ satisfy $q<\frac{i(\fai)}{r}$.
Then, by the lower type $qr$ property of $\fai(x,\cdot)$, we know that
\begin{eqnarray*}
\dsup_{t\in(0,\,\fz)}\lf\|\lf|F_1(\cdot,\,t)\r|^q\r\|^r_{L^r_{\fai(\cdot,\,1)}(\rn)}&&
=\dsup_{t\in(0,\,\fz)}\lf\{\dint_{\rn} \lf| F_1(x,\,t)\r|^{qr}\fai(x,\,1)\,dx\r\}\\
&&\le \dsup_{t\in(0,\,\fz)}\lf\{\dint_{\rn} \fai(x,\, |F_1(x,\,t)|)\,dx\r\}\\
&&\le \dsup_{t\in(0,\,\fz)}\lf\{\dint_{\rn} \fai(x,\, |F(x,\,t)|)\,dx\r\}<\fz.
\end{eqnarray*}
Thus, $\{|F_1(\cdot,\,t)|^q\}_{t>0}$ is uniformly bounded in $L^r_{\fai(\cdot,\,1)}(\rn)$,
which, together with the weak compactness of $L^r_{\fai(\cdot,\,1)}(\rn)$, implies that
there exist $\wz{h}_1\in L^r_{\fai(\cdot,\,1)}(\rn)$ and a subsequence
$\{|F_1(\cdot,\,t_k)|^q\}_{k\in\nn}$ such that $t_k\to 0^+$
and $\{|F_1(\cdot,\,t_k)|^q\}_{k\in\nn}$
converges weakly to $\wz{h}_1$ in $L^r_{\fai(\cdot,\,1)}(\rn)$ as $k\to \fz$, namely,
for any $g\in L^{r'}_{\fai(\cdot,\,1)}(\rn)$,
\begin{eqnarray}\label{eqn est1 lem MOHVBV}
\lim_{k\to \fz} \dint_\rn \lf|F_1(y,\,t_k)\r|^qg(y)\fai(y,\,1)\,dy= \dint_{\rn}
\wz{h}_1(y)g(y)\fai(y,\,1)\,dy.
\end{eqnarray}
Now, for all $y\in\rn$, let
\begin{eqnarray}\label{eqn est2 lem MOHVBV}
g(y):=\frac{P_t(x-y)}{\fai(y,\,1)},
\end{eqnarray}
where $P_t$ is the Poisson kernel as in \eqref{eqn PK}. By using the $B_{r'}(\rn)$-condition as in
\eqref{eqn Br condition} and $\fai(\cdot,\,1)\in A_r(\rn)$, we conclude that
\begin{eqnarray*}
\dint_{\rn}\lf|g(y)\r|^{r'}\fai(y,\,1)\,dy&&=\dint_{\rn} \lf[\frac{P_t(x-y)}
{\fai(y,\,1)}\r]^{r'}\fai(y,\,1)\,dy\\
&&\ls \dint_{\rn} \frac{1} {(t+|x-y|)^{nr'}} \lf[\fai(y,\,1)\r]^{-r'/r}\,dy\\
&&\ls \dint_{B(x,\,t)} \lf[\fai(y,\,1)\r]^{-r'/r}\,dy<\fz,
\end{eqnarray*}
which implies that $g\in L^{r'}_{\fai(\cdot,\,1)}(\rn)$.
Thus, from \eqref{eqn est1 lem MOHVBV}, we deduce that, for all $(x,\,t)\in\rr^{n+1}_+$,
\begin{eqnarray}\label{eqn est3 lem MOHVBV}
\lim_{k\to \fz} \lf(\lf|F_1(\cdot,\,t_k)\r|^q*P_t\r)(x)=\lf(\wz{h}_1*P_t\r)(x).
\end{eqnarray}

On the other hand, since $\sup_{(x,\,t)\in\rr^{n+1}_+}\lf|F_2(x,\,t)\r|^q\le 1$, we know
that  $\{|F_2(\cdot,\,t)|^q\}_{t>0}$ is uniformly bounded in $L^\fz_{\fai(\cdot,\,1)}(\rn)$.
Thus, there exist $\wz{h}_2\in L^\fz_{\fai(\cdot,\,1)}(\rn)$ with $\|\wz{h}_2\|_{L^\fz_{\fai(\cdot,\,1)}(\rn)}\le 1$ and a subsequence
$\{|F_2(\cdot,\,t_k)|^q\}_{k\in\nn}$ such that $t_k\to 0^+$ and
$\{|F_2(\cdot,\,t_k)|^q\}_{k\in\nn}$ converges $\ast$-weakly to $\wz{h}_2$ in
$L^\fz_{\fai(\cdot,\,1)}(\rn)$ as $k\to \fz$, namely,
for any $g\in L^{1}_{\fai(\cdot,\,1)}(\rn)$,
\begin{eqnarray}\label{eqn est1-1 lem MOHVBV}
\lim_{k\to \fz} \dint_\rn \lf|F_2(y,\,t_k)\r|^qg(y)\fai(y,\,1)\,dy= \dint_{\rn}
\wz{h}_2(y)g(y)\fai(y,\,1)\,dy.
\end{eqnarray}
Here, by abuse of notation, we use the same subscripts for the above
two different subsequences in our arguments.

Let $g$ be as in \eqref{eqn est2 lem MOHVBV}. It is easy to see that
$\int_{\rn}g(y)\fai(y,\,1)\,dy=1.$
Thus, by \eqref{eqn est1-1 lem MOHVBV}, we find that, for all $x\in\rn$,
\begin{eqnarray}\label{eqn est4 lem MOHVBV}
\lim_{k\to \fz} \lf(\lf|F_2(\cdot,\,t_k)\r|^q*P_t\r)(x)=\lf(\wz{h}_2*P_t\r)(x).
\end{eqnarray}

Now, let $\wz{h}:=\wz{h}_1+\wz{h}_2$.
Observe that, for all $k\in\nn$, $|\supp (F_1(\cdot,\,t_k))
\cap \supp (F_2(\cdot,\,t_k))|=0$,
which further implies that $|\supp(\wz{h}_1)\cap \supp(\wz{h}_2)|=0$. Moreover,
from \eqref{eqn est3 lem MOHVBV} and \eqref{eqn est4 lem MOHVBV}, it follows that, for
all $x\in\rn$,
\begin{eqnarray}\label{eqn est8 lem MOHVBV}
\lim_{k\to \fz} \lf(\lf|F(\cdot,\,t_k)\r|^q*P_t\r)(x)=\lf(\wz{h}*P_t\r)(x).
\end{eqnarray}
This, together with $t_k\to 0^+$ as $k\to \fz$ and
Lemma \ref{lem harmonic majorant of MOHVH}, shows that,
for all $(x,\,t)\in\rr^{n+1}_+$,
\begin{eqnarray*}
\lf|F(x,\,t)\r|^q=\lim_{k\to \fz}\lf|F(x,\,t+t_k)\r|^q
\le \lim_{k\to\fz}\lf(\lf|F(\cdot,\,t_k)\r|^q*P_t\r)(x)=
\lf(\wz{h}*P_t\r)(x),
\end{eqnarray*}
which proves \eqref{eqn harmoic marjant BV} by taking $h:=\wz{h}^{1/q}$.

Now, we prove that $\wz{h}$ is the non-tangential limit of $|F(\cdot,\,t_k)|^q$.
Using \eqref{eqn est3 lem MOHVBV} and \eqref{eqn est4 lem MOHVBV}, we conclude that,
for all $x\in\rn$,
\begin{eqnarray}\label{eqn est5 lem MOHVBV}
\quad\quad|F|^*(x)&&\ls\lf[\dsup_{|y-x|<t,\,t\in(0,\,\fz)}
\lf(\wz{h}_1*P_t\r)(y)\r]^{1/q} +\lf[\dsup_{|y-x|<t,\,t\in(0,\,\fz)}
\lf(\wz{h}_2*P_t\r)(y)\r]^{1/q} \\
&&\nonumber=:\mathrm{I}+\mathrm{II}.
\end{eqnarray}

To estimate $\mathrm{II}$, from the fact that
$\|\wz{h}_2\|_{L^\fz_{\fai(\cdot,\,1)}(\rn)}\le 1$ and \eqref{eqn db property},
we deduce that $\|\wz{h}_2\|_{L^\fz(\rn)}\le 1$. This implies that
\begin{eqnarray}\label{eqn est6 lem MOHVBV}
\mathrm{II}\ls \lf[\dsup_{|y-x|<t,\,t\in(0,\,\fz)}\dint_{\rn} P_t(y)\,dy\r]^{1/q}\ls 1.
\end{eqnarray}

For $\mathrm{I}$, it is easy to see that
\begin{eqnarray}\label{eqn est7 lem MOHVBV}
\mathrm{I}\sim
\lf[\lf(\wz{h}_1\r)_{P}^*(x)\r]^{1/q}.
\end{eqnarray}
Moreover, by the fact that $\fai(\cdot,\,1)\in A_r(\rn)$, $r\in (q(\fai),\,\fz)$ and the
boundedness of the Hardy-Littlewood maximal function $\mathcal{M}$
on $L^r_{\fai(\cdot,\,1)}(\rn)$ (see, for example, \cite[Theorem 9.1.9]{Gr09}),
we conclude that
\begin{eqnarray*}
\dint_{\rn}\lf[\lf(\wz{h}_1\r)_{P}^*(x)\r]^r\fai(x,\,1)\,dx&&\ls
\dint_{\rn}\lf[\mathcal{M}\lf(\wz{h}_1\r)(x)\r]^r\fai(x,\,1)\,dx\\
&&\ls\dint_{\rn}\lf[\wz{h}_1(x)\r]^r\fai(x,\,1)\,dx<\fz,
\end{eqnarray*}
which, together with \eqref{eqn est5 lem MOHVBV},
\eqref{eqn est6 lem MOHVBV} and \eqref{eqn est7 lem MOHVBV},
implies that, for almost every $x\in\rn$, $|F|^*(x)<\fz$.
From the fact that each coordinate
function of $F$ is harmonic on $\rr^{n+1}_+$ and Fatou's theorem
(see \cite[p.\,47]{St70}), we deduce that $F(x,\,t)$ has a non-tangential limit
as $t\to 0^+$, which, combined with the uniqueness of the limit, implies that
$\wz{h}$ is the non-tangential limit of $|F(\cdot,\,t)|^q$ as $t\to 0^+$.

Now, we show that $h\in L^\fai(\rn)$ by using some properties of convex
Musielak-Orlicz spaces from \cite{DHHR11}. For all $(x,\,t)\in \rr^{n+1}_+$, let
\begin{eqnarray}\label{2.y2}
\fai_q(x,\,t):=\fai(x,\,t^{1/q}).
\end{eqnarray}
By an elementary calculation, we see that
\begin{eqnarray}\label{eqn scaling of MOF}
\lf\|\lf|F(\cdot,\,t)\r|\r\|_{L^\fai(\rn)}=
\lf\|\lf|F(\cdot,\,t)\r|^q\r\|_{L^{\fai_q}(\rn)}^{1/q},
\end{eqnarray}
which, together with the fact that $F\in \mathcal{H}_\fai(\rr^{n+1}_+)$,
implies that $\{|F(\cdot,\,t)|^q\}_{t\in(0,\,\fz)}$ is uniformly bounded in $L^{\fai_q}(\rn)$.

Moreover, using the fact that $i(\fai_q)=\frac{i(\fai)}{q}>q(\fai)\ge1$ and
\begin{eqnarray*}
\fai_q(x,\,t)\sim \dint_{0}^t \frac{\fai_q(x,\,s)}{s}\,ds,
\end{eqnarray*}
we conclude that $\wz{\fai}_q(x,\,s):=\fai_q(x,\,s)/s$ satisfies the following properties:
\begin{itemize}
\vspace{-0.25cm}
\item[(i)] $\lim_{s\to 0^+} \wz{\fai}_q(x,\,s)=0$, $\lim_{t\to \fz}\wz{\fai}_q(x,\,s)=\fz$
and, when $s\in(0,\,\fz)$, $\wz{\fai}_q(x,\,s)>0$;
\vspace{-0.25cm}
\item[(ii)] for all $x\in\rn$, $\wz{\fai}_q(x,\,\cdot)$ is decreasing;
\vspace{-0.25cm}
\item[(iii)] for all $x\in\rn$, $\wz{\fai}_q(x,\,\cdot)$ is right continuous.
\vspace{-0.25cm}
\end{itemize}
Thus, from \cite[p.\,262]{Ad03}, we deduce that, for all $x\in\rn$,  $\fai_q(x,\,\cdot)$ is
equivalent to an $N$-function (see \cite{Ad03} for the definition of $N$-functions).
Hence, by \cite[p.\,38, Theorem 2.3.13]{DHHR11}, we know  that $L^{\fai_q}(\rn)$ is a
Banach space.

For all $(x,\,t)\in\rr^{n+1}_+$, let
\begin{eqnarray*}
\fai_q^*(x,\,t):=\dsup_{s\in(0,\,\fz)}\lf\{st-\fai_q(x,\,s)\r\}.
\end{eqnarray*}
It follows, from \cite[p.\,59]{DHHR11}, that $L^{\fai_q}(\rn)\subset
\lf(L^{\fai_q^*}(\rn)\r)^*$. Thus, by Alaoglu's theorem, we obtain the
$\ast$-weak compactness of $L^{\fai_q}(\rn)$, which, together with the fact
$\{|F(\cdot,\,t)|^q\}_{t\in(0,\,\fz)}$ is uniformly bounded in $L^{\fai_q}(\rn)$, implies
that there exist $g\in L^{\fai_q}(\rn)$ and a subsequence $\{|F(\cdot,\,t_k)|^q\}_{k\in\nn}$
such that $t_k\to 0^+$ and $\{|F(\cdot,\,t_k)|^q\}_{k\in\nn}$ converges $\ast$-weakly
to $g$ in $ L^{\fai_q}(\rn)$ as $k\to \fz$.
Moreover,  from the uniqueness of the limit, we deduce that, for almost every $x\in\rn$,
$h(x)=[g(x)]^{1/q}$. Thus, $h\in L^\fai(\rn)$.

The formula, $\lim_{t\to 0^+}|F(\cdot,\,t)|=h(\cdot)$ in $L^\fai(\rn)$, follows immediately
from the facts that $h$ is the non-tangential limit of $F$ as $t\to 0^+$ almost everywhere,
$F\in \mathcal{H}_\fai(\rr^{n+1}_+)$, $h\in L^\fai(\rn)$ and the dominated convergence theorem.
This finishes the proof of Lemma \ref{lem boudnary value of MOHVH}.
\end{proof}

\begin{remark}\label{rem range of q}
We point out that, in the proofs of Lemmas \ref{lem harmonic majorant of MOHVH} and
\ref{lem boudnary value of MOHVH}, we used the condition  $\frac{i(\fai)}{q(\fai)}
> \frac{n-1}{n}$ merely because we need the fact that $|F|^q$ is subharmonic
on $\rr^{n+1}_+$ for $q\in[\frac{n-1}{n},\,\fz)$. Thus, if there exists $q$,
whose size is strictly less than $\frac{n-1}{n}$,
such that $|F|^q$ is subharmonic on $\rr^{n+1}_+$, then, for all
$\frac{i(\fai)}{q(\fai)}>q$, the conclusions of Lemmas
\ref{lem harmonic majorant of MOHVH} and \ref{lem boudnary value of MOHVH}
still hold true. Moreover, if, for all $q\in(0,\,\fz)$,
$|F|^q$ is subharmonic on $\rr^{n+1}_+$, then, for every Musielak-Orlicz function $\fai$
satisfying Assumption $(\fai)$, by taking $q$ sufficiently small, we see that
$\frac{i(\fai)}{q(\fai)}>q$ always holds true. Thus, in this case,
Lemmas \ref{lem harmonic majorant of MOHVH} and
\ref{lem boudnary value of MOHVH} hold true for every Musielak-Orlicz function
satisfying Assumption $(\fai)$.
\end{remark}

With these preparations, we now turn to the study of the relation between
$\mathcal{H}_\fai(\rr^{n+1}_+)$ and $H_\fai(\rr^{n+1}_+)$.

\begin{proposition}\label{pro HVC}
Let $\fai$ satisfy Assumption $(\fai)$ with
$\frac{i(\fai)}{q(\fai)}>\frac{n-1}{n}$ and
$$F:=\{u_0,\,u_1,\,\ldots,\,u_n\}\in
\mathcal{H}_{\fai}(\rr^{n+1}_+),$$
where $i(\fai)$ and $q(\fai)$ are as in \eqref{1.6}
and \eqref{1.7}, respectively. Then there exists a harmonic function
$u:= u_0\in H_\fai(\rr^{n+1}_+)$  such that
\begin{eqnarray}\label{eqn est2 pro HVtoHF}
\lf\|u\r\|_{H_\fai(\rr^{n+1}_+)}\le C \lf\|F\r\|_{\mathcal{H}_\fai(\rr^{n+1}_+)},
\end{eqnarray}
where $C$ is a positive constant independent of $u$ and $F$.
\end{proposition}

\begin{proof}
Let $F\in \mathcal{H}_\fai(\rr^{n+1}_+)$. By Lemmas \ref{lem harmonic majorant of MOHVH}
and \ref{lem boudnary value of MOHVH}, we see that $|F|$ has the non-tangential limit
$F(\cdot,\,0)$. Moreover, for all $(x,\,t)\in\rr^{n+1}_+$,
 \begin{eqnarray}\label{eqn est1 pro HVtoHF}
\lf|F(x,\,t)\r|^q \le \lf(\lf|F(\cdot,\,0)\r|^q*P_t\r)(x)\ls \mathcal{M}\lf(\lf|
F(\cdot,\,0)\r|^q\r)(x),
\end{eqnarray}
where $q\in[\frac{n-1}{n},\, \frac{i(\fai)}{q(\fai)})$ is as in Lemma
\ref{lem boudnary value of MOHVH} and $\mathcal{M}$ denotes
the Hardy-Littlewood maximal function.
Let $u:=u_0$ and $\fai_q$ be as in \eqref{2.y2}.
For all $\lz\in(0,\,\fz)$, from \eqref{eqn est1 pro HVtoHF}, the fact
that $q(\fai)<\frac{i(\fai)}{q}$ and Lemma \ref{lem bd HLMF}, it follows that
\begin{eqnarray*}
\dint_{\rn}\fai\lf(x,\,\frac{u^*(x)}{\lz}\r)\,dx&&\le
\dint_{\rn}\fai_q\lf(x,\,\frac{(|F|^q)^*(x)}{\lz^q}\r)\,dx \\
&&\ls \dint_{\rn}\fai_q\lf(x,\,\frac{(\mathcal{M}(|F(\cdot,\,0)|^q))^*(x)}{\lz^q}\r)\,dx\\
&&\ls \dint_{\rn}\fai_q\lf(x,\,\frac{|F(x,\,0)|^q}{\lz^q}\r)\,dx\\
&&\sim \dint_{\rn}\fai\lf(x,\,\frac{|F(x,\,0)|}{\lz}\r)\,dx \ls \sup_{t\in(0,\,\fz)}
\dint_{\rn}\fai\lf(x,\,\frac{|F(x,\,t)|}{\lz}\r)\,dx,
\end{eqnarray*}
which immediately implies \eqref{eqn est2 pro HVtoHF} and
hence completes the proof of Proposition \ref{pro HVC}.
\end{proof}

Proposition \ref{pro HVC} immediately implies the following conclusion,
the details being omitted.

\begin{corollary}\label{cor HVC}
Let $\fai$ satisfy Assumption $(\fai)$ with
$\frac{i(\fai)}{q(\fai)}>\frac{n-1}{n}$ and
$$F:=\{u_0,\,u_1,\,\ldots,\,u_n\}\in
\mathcal{H}_{\fai,\,2}(\rr^{n+1}_+),$$
where $i(\fai)$ and $q(\fai)$ are as in \eqref{1.6}
and \eqref{1.7}, respectively. Then there exists a harmonic function
$u:= u_0\in H_{\fai,\,2}(\rr^{n+1}_+)$  such that
\begin{eqnarray*}
\lf\|u\r\|_{H_\fai(\rr^{n+1}_+)}\le C \lf\|F\r\|_{\mathcal{H}_\fai(\rr^{n+1}_+)},
\end{eqnarray*}
where $C$ is a positive constant independent of $u$ and $F$.
\end{corollary}

Furthermore, we have the following relation between $H_\fai(\rn)$ and
$\mathcal{H}_{\fai,\,2}(\rr^{n+1}_+)$, which also implies that $H_\fai(\rn)$
consists of the boundary values of real parts of
$\mathcal{H}_{\fai,\,2}(\rr^{n+1}_+)$.

\begin{proposition}\label{pro MOHtoMOHVH}
Let $\fai$ satisfy Assumption $(\fai)$ and $f\in H_\fai(\rn)$.
Then there exists $F:=\{u_0,\,u_1,\,\ldots,\,u_n\}\in \mathcal{H}_{\fai,\,2}
(\rr^{n+1}_+)$ such that $F$ satisfies the generalized Cauchy-Riemann
equation \eqref{eqn GCR equation} and that, for all
$(x,\,t)\in\rr^{n+1}_+$, $u_0(x,\,t):=(f*P_t)(x)$,
where $P_t$ is the Poisson kernel as in \eqref{eqn PK}. Moreover,
\begin{eqnarray}\label{eqn pro MOHtoMOHVH}
\lf\|F\r\|_{\mathcal{H}_\fai(\rr^{n+1}_+)}\le C \|f\|_{H_\fai(\rn)},
\end{eqnarray}
where $C$ is a positive constant independent of $f$ and $F$.
\end{proposition}

\begin{proof}
Let $f\in H_\fai(\rn)$. By Remark \ref{rem defMOH}(iii),
we see that $L^2(\rn)\cap H_\fai(\rn)$ is dense in $H_\fai(\rn)$.
Thus,  there exists a sequence $\{f_k\}_{k\in\nn}\subset (L^2(\rn)\cap H_\fai(\rn))$
such that $\lim_{k\to \fz} f_k =f$ in $H_\fai(\rn)$ and hence in $\mathcal{S}'(\rn)$.

For any $k\in\nn$, $j\in\{1,\,\ldots,\,n\}$ and $(x,\,t)\in\rr^{n+1}_+$,
let $u^k_0(x,\,t):=(f_k*P_t)(x)$ and $u^k_j(x,\,t):=(f_k*Q^{(j)}_t)(x)$,
where  $P_t$ is the Poisson kernel as in \eqref{eqn PK}
and $Q_t^{(j)}$ the $j$-\emph{th conjugate Poisson kernel} defined by setting,
for all $x\in\rn$,
\begin{eqnarray}\label{eqn CPK}
Q_t^{(j)}(x):=C_{(n)} \frac{x_j}{(t^2+|x|^2)^{\frac{n+1}{2}}},
\end{eqnarray}
where $C_{(n)}$ is as in \eqref{eqn RZ}.

Since $f_k\in L^2(\rn)$, we deduce, from \cite[p.\,236, Theorem 4.17]{SW71}, that the
harmonic vector $F_k:=\{u^k_0,\,u^k_1,\,\ldots,\,u^k_n\}\in
\mathcal{H}^2(\rr^{n+1}_+)$ and satisfies the generalized
Cauchy-Riemann equation \eqref{eqn GCR equation}. Moreover, by using the Fourier transform,
we see that, for all $j\in\{0,\,1,\,\ldots,\,n\}$ and $(x,\,t)\in\rr^{n+1}_+$,
$(Q^{(j)}_t*f_k)(x)=(R_j(f_k)*P_t)(x)$ (see also \cite[p.\,65, Theorem 3]{St70}),
which, together with
Proposition \ref{pro PIC} and the boundedness of $R_j$ on $H_\fai(\rn)$
(see Corollary \ref{cor bd of Riesz} below), implies that,
for all $j\in\{0,\,1,\,\ldots,\,n\}$,
\begin{eqnarray*}
\dsup_{t\in(0,\,\fz)}\lf\|\lf|u_j^{k}(\cdot,\,t)\r|\r\|_{L^\fai(\rn)}\ls
 \|R_j(f_k)\|_{H_\fai(\rn)}\ls \|f_k\|_{H_\fai(\rn)}\ls \|f\|_{H_\fai(\rn)}.
\end{eqnarray*}
Thus,
\begin{eqnarray}\label{eqn est1 pro MOHtoMOHVH}
\dsup_{t\in(0,\,\fz)}\lf\|F_k(\cdot,\,t)\r\|_{L^\fai(\rn)}\ls\|f\|_{H_\fai(\rn)}<\fz,
\end{eqnarray}
which implies that $F_k\in \mathcal{H}_\fai(\rr^{n+1}_+)$ and hence
$F_k\in \mathcal{H}_\fai(\rr^{n+1}_+)\cap \mathcal{H}^2(\rr^{n+1}_+)$.

We point out that, in the above argument, we used the boundedness of the Riesz transform
$R_j$ on $H_\fai(\rn)$, which will be proved in Corollary \ref{cor bd of Riesz} below,
whose proof does not use the conclusion of Proposition \ref{pro MOHtoMOHVH}.
So, there exists no risk of circular reasoning.

On the other hand, from $\lim_{k\to \fz}f_k=f$ in $H_\fai(\rn)$ and hence in
$\mathcal{S}'(\rn)$, $$\lim_{k\to \fz} R_j(f_k)=
R_j(f)$$ in $H_\fai(\rn)$ and hence in
$\mathcal{S}'(\rn)$, and \eqref{eqn decomposition PK}, we deduce that,
for all $(x,\,t)\in\rr^{n+1}_+$,
$\lim_{k\to\fz} f_k*P_t(x)=f*P_t(x)$ and $\lim_{k\to\fz} R_j(f_k)*P_t(x)=R_j(f)*P_t(x)$.

Now, we claim that the above two limits are uniform on compact sets.
Indeed, for all $(x,\,t)\in\rr^{n+1}_+$, $y,\,z\in B(x,\,\frac{t}{4})$,
$\wz t\in(\frac{3t}{4},\,\frac{5t}{4})$
and $\phi\in\mathcal{S}(\rn)$ satisfies $\int_\rn \phi(x)\,dx=1$, by the
definition of the non-tangential maximal function, we know that
\begin{eqnarray}\label{eqn est2 pro MOHtoMOHVH}
\lf|\lf(\lf[f_k-f\r]*\phi_{\wz t}\r)(z)\r|\le \mathcal{M}_\phi^*\lf(f_k-f\r)(y).
\end{eqnarray}

Moreover, for any $\epsilon\in(0,\,1]$ and $q\in(I(\fai),\,\fz)$, from the upper type $q$ property of $\fai(x,\cdot)$, it follows that
\begin{eqnarray*}
\epsilon^q \dint_{\{x\in\rn:\ \mathcal{M}_\phi^*\lf(f_k-f\r)(x)>\epsilon\}} \fai\lf(x,\,1\r)
\,dx&&\ls \dint_{\{x\in\rn:\ \mathcal{M}_\phi^*\lf(f_k-f\r)(x)>\epsilon\}}
\fai\lf(x,\,\epsilon\r)\,dx\\
&& \ls \dint_{\rn} \fai\lf(x,\,\mathcal{M}_\phi^*\lf(f_k-f\r)(x)\r)\,dx,
\end{eqnarray*}
which tends to $0$ as $k\to \fz$. Thus, $\mathcal{M}_\phi^*\lf(f_k-f\r)$ converges
to $0$ in the measure $\fai(\cdot,\,1)\,dx$. This shows that there exists $k_0\in\nn$
such that, for all $k\in\nn$ with $k\ge k_0$,
\begin{eqnarray}\label{eqn est3 pro MOHtoMOHVH}
\dint_{B(x,\,\frac{t}{4})}\fai(y,\,1)\,dy\ge 2\dint_{\mathrm{E}_k^\complement}
\fai\lf(y,\,1\r)\,dy,
\end{eqnarray}
where $\mathrm{E}_k:=\{y\in B(x,\,\frac{t}{4}):\
\mathcal{M}_\phi^*\lf(f_k-f\r)(y)< 1\}$.

Combined \eqref{eqn est2 pro MOHtoMOHVH} with \eqref{eqn est3 pro MOHtoMOHVH} and the upper type
1 property of $\fai(x,\cdot)$,
we conclude that, for all $z\in B(x,\,\frac{t}{4})$ and $\wz{t}\in (\frac{3t}{4},\,\frac{5t}{4})$,
\begin{eqnarray*}
\lf|\lf(\lf[f_k-f\r]*\phi_{\wz t}\r)(z)\r|&&\le \frac{1}{\int_{\mathrm{E}_k}\fai(y,\,1)\,dy}
\dint_{\mathrm{E}_k} \mathcal{M}_\phi^*\lf(f_k-f\r)(y) \fai(y,\,1)\,dy\\
&&\ls \frac{1}{\int_{B(x,\,\frac{t}{4})}\fai(y,\,1)\,dy}
\dint_{\mathrm{E}_k}  \fai(y,\,\mathcal{M}_\phi^*\lf(f_k-f\r)(y))\,dy,
\end{eqnarray*}
which tends to $0$ as $k\to \fz$. This implies that $f_k*\phi_t$
converges uniformly to $f*\phi_t$ on $B(x,\,\frac{t}{4})\times (\frac{3t}{4},\,\frac{5t}{4})$.

Moreover, using \eqref{eqn decomposition PK}, we know
$\lim_{k\to\fz} f_k*P_t(x)=f*P_t(x)$ uniform on compact sets.
Similarly, we also conclude $\lim_{k\to\fz} R_j(f_k)*P_t(x)=R_j(f)*P_t(x)$
uniform on compact sets. This shows the above claim.

By the above claim and the fact that $F_k$ satisfies the generalized
Cauchy-Riemann equation \eqref{eqn GCR equation}, we know that $F:=\{f*P_t,\,R_1(f)*P_t,\,\ldots,\,R_n(f)*P_t\}$ also satisfies the
generalized Cauchy-Riemann equation \eqref{eqn GCR equation}, which, together with
Fatou's lemma and \eqref{eqn est1 pro MOHtoMOHVH}, implies that
\begin{eqnarray*}
\dsup_{t\in(0,\,\fz)}\lf\|\lf|F(\cdot,\,t)\r|\r\|_{L^\fai(\rn)}&&=\dsup_{t\in(0,\,\fz)}
\lf\|\lim_{k\to\fz}\lf|F_k(\cdot,\,t)\r|\r\|_{L^\fai(\rn)}\\
&&\le\dsup_{t\in(0,\,\fz)}
\mathop{\underline{\lim}}_{k\to \fz} \lf\|\lf|F_k(\cdot,\,t)\r|\r\|_{L^\fai(\rn)}
\ls \|f\|_{H_\fai(\rn)}<\fz.
\end{eqnarray*}
Thus, $F\in \mathcal{H}_{\fai,\,2}(\rr^{n+1}_+)$ and \eqref{eqn pro MOHtoMOHVH}
holds true,
which completes the proof of Proposition \ref{pro MOHtoMOHVH}.
\end{proof}

Combined Propositions \ref{pro HFC}, \ref{pro HVC} with \ref{pro MOHtoMOHVH}, we
immediately obtain the following conclusion.

\begin{theorem}\label{cor equiv 3spaces}
Let $\fai$ satisfy Assumption $(\fai)$ with $\frac{i(\fai)}{q(\fai)}>
\frac{n-1}{n}$, where $i(\fai)$
and $q(\fai)$ are as in \eqref{1.6}
and \eqref{1.7}, respectively. Then the spaces $H_\fai(\rn)$,
$H_{\fai,\,2}(\rr^{n+1}_+)$ and  $\mathcal{H}_{\fai,\,2}(\rr^{n+1}_+)$,
defined, respectively, in Definitions \ref{def MOH space}, \ref{def MOHH space p>1}
and \ref{def MOHVH space}, are isomorphic to each other.

More precisely, the following statements hold
true:
\begin{itemize}
\vspace{-0.25cm}
\item[{\rm (i)}] $u\in H_{\fai,\,2}(\rr^{n+1}_+)$ if and only
if there exists $f\in H_\fai(\rn)$ such that, for all $(x,\,t)\in\rr^{n+1}_+$,
$u(x,\,t)=(f*P_t)(x)$,
where $P_t$ is the Poisson kernel as in \eqref{eqn PK}.
\vspace{-0.25cm}
\item[{\rm (ii)}]  If $F:=(u_0,\,u_1,\,\ldots,\,u_n)\in
\mathcal{H}_{\fai,\,2}(\rr^{n+1}_+)$, then $u_0\in H_{\fai,\,2}(\rr^{n+1}_+)$.
\vspace{-0.25cm}
\item[{\rm (iii)}]  If $f\in
{H}_{\fai}(\rn)$, then there exists
$F:=\{u_0,\,u_1,\,\ldots,\,u_n\}\in \mathcal{H}_{\fai,\,2}(\rr^{n+1}_+)$ such
that, for all $(x,\,t)\in\rr^{n+1}_+$, $u_0(x,\,t):=f*P_t(x)$.
\end{itemize}
\end{theorem}

\subsection{First order Riesz transform characterizations}
\label{s24}

\hskip\parindent In this subsection,
we give out the proof of Theorem \ref{thm1}.
To this end, we first give a sufficient condition on operators to be bounded on $H_\fai(\rn)$.
We now recall the notion of $H_\fai(\rn)$-atoms introduced in \cite[Definition 2.4]{K11} as follows.
\begin{definition}[\cite{K11}]\label{def MOH atom}
Let $\fai$ satisfy Assumption $(\fai)$,
$q\in(q(\fai),\fz]$ and $s\in\zz_+$ satisfy $s\ge\lfz
n[\frac{q(\fai)}{i(\fai)}-1]\rfz$, where
$q(\fai)$ and $i(\fai)$ are as in \eqref{1.7}
and \eqref{1.6}, respectively. A measurable function $a$ on
$\rn$ is called a \emph{$(\fai,\,q,\,s)$-atom}, if there exists a ball
$B\subset\rn$ such that

$\mathrm{(i)}$ $\supp a\subset B$;

$\mathrm{(ii)}$
$\|a\|_{L^q_{\fai}(B)}\le\|\chi_B\|_{L^\fai(\rn)}^{-1}$,
where
\begin{equation*}
\|a\|_{L^q_{\fai}(B)}:=
\begin{cases}\dsup_{t\in (0,\fz)}
\lf[\frac{1}
{\fai(B,t)}\dint_{\rn}|a(x)|^q\fai(x,t)\,dx\r]^{1/q},& q\in [1,\fz),\\
\|a\|_{L^{\fz}(B)},&q=\fz,
\end{cases}
\end{equation*}
and $\fai(B,\,t):=\int_{B}\fai(x,\,t)\,dx$;

$\mathrm{(iii)}$ $\int_{\rn}a(x)x^{\az}\,dx=0$ for all
$\az:=(\az_1,\,\ldots,\,\az_n)\in\zz_+^n$ with $|\az|:=
\az_1+\dots+\az_n\le s$.
\end{definition}

Let $T$ be a sublinear operator.  Recall that $T$ is said to be \emph{nonnegative}
if, for all $f$ in the domain of $T$, $Tf \ge0$ .

\begin{lemma}\label{lem sufficent COMHS}
Let $\fai$ satisfy Assumption $(\fai)$ and
$s\in \zz_+$ satisfy $s\ge m(\fai):=\lfloor n(\frac{q(\fai)}{i(\fai)}-1) \rfloor$,
where $i(\fai)$ and $q(\fai)$ are as in \eqref{1.6}
and \eqref{1.7}, respectively.
Suppose that $T$ is a linear
(resp. nonnegative sublinear) operator, which is of weak type $(L^2(\rn),\,L^2(\rn))$.
If there exists a positive constant $C$ such that, for every
$\lz\in\cc$ and $(\fai,\,q,\,s)$-atom $a$ associated with the ball $B$,
\begin{eqnarray}\label{eqn sufficient CD}
\dint_{\rn}\fai\lf(x,\,T\lf(\lz a\r)(x)\r)\,dx\le C \dint_{B}\fai
\lf(x,\,\frac{|\lz|}{\|\chi_B\|_{L^\fai(\rn)}}\r)\,dx,
\end{eqnarray}
then $T$ can be extended to a bounded linear (resp. nonnegative sublinear) operator
from $H_\fai(\rn)$ to $L^\fai(\rn)$.
\end{lemma}

\begin{proof}
Lemma \ref{lem sufficent COMHS} is a special case of  \cite[Lemma 5.6]{yys} when
the operator $L$ considered therein is the Laplace operator
$-\bdz$. The only difference is that here we use the
$(\fai,\,q,\,s)$-atoms to replace the operator-adapted atoms therein,
the details being omitted. This finishes the proof of Lemma \ref{lem sufficent COMHS}.
\end{proof}

Using Lemma \ref{lem sufficent COMHS}, we establish the following proposition of
the interpolation of operators.

\begin{proposition}\label{pro interpolation}
Let $\fai$ satisfy Assumption $(\fai)$,
$I(\fai)$ and $i(\fai)$ be as in \eqref{1.5} and \eqref{1.6}, respectively.
Assume that $T$ is a linear (resp. nonnegative sublinear)
operator and either of the following two conditions holds true:
\vspace{-0.25cm}
\begin{itemize}
\item[{\rm (i)}] if $0<p_1<i(\fai)\le I(\fai)\le 1<p_2<\fz$ and, for all $t\in(0,\,\fz)$,
$T$ is of weak type $(H_{\fai(\cdot,\,t)}^{p_1}(\rn),\,L_{\fai(\cdot,\,t)}^{p_1}(\rn))$
and  of weak type  $(L_{\fai(\cdot,\,t)}^{p_2}(\rn),\,L_{\fai(\cdot,\,t)}^{p_2}(\rn))$;
\vspace{-0.25cm}
\item[{\rm (ii)}] if $0<p_1<i(\fai)\le I(\fai)<p_2\le 1$ and, for all $t\in(0,\,\fz)$,
$T$ is of weak type $(H_{\fai(\cdot,\,t)}^{p_1}(\rn),\,L_{\fai(\cdot,\,t)}^{p_1}(\rn))$
and  of weak type  $(H_{\fai(\cdot,\,t)}^{p_2}(\rn),\,L_{\fai(\cdot,\,t)}^{p_2}(\rn))$.
\end{itemize}
\vspace{-0.25cm}
Then $T$ is bounded from $H_\fai(\rn)$ to $L^\fai(\rn)$.
\end{proposition}

\begin{proof}
Assume first that (i) holds true. Let $q\in(\max\{q(\fai),\,p_2\},\fz)$, $s\in\zz_+$
satisfy $s\ge\lfz n(\frac{q(\fai)}{p_1}-1)\rfz$ with $q(\fai)$ as in \eqref{1.7},
$\lz\in(0,\,\fz)$ and $a$ be a $(\fai,\,q,\,s)$-atom associated
with the ball $B$. From the fact that $T$ is of weak type $(L_{\fai(\cdot,\,t)}^{p_2}(\rn),\,L_{\fai(\cdot,\,t)}^{p_2}(\rn))$,
Definition \ref{def MOH atom}(ii) and H\"older's inequality,
it follows that, for all $\az\in(0,\,\fz)$,
\begin{eqnarray}\label{eqn est1 pro inter}
\qquad\dint_{\{x\in\rn:\ |T(\lz a)(x)|>\az\}}\fai\lf(x,\,t\r)\,dx
&&\ls \frac{1}{\az^{p_2}}
\dint_{\rn} \lf|\lz a(x)\r|^{p_2}\fai(x,\,t)\,dx\\
&&\sim \frac{\lz^{p_2}}{\az^{p_2}}
\lf[\frac1{\fai(B,\,t)}{\int_{\rn} \lf|a(x)\r|^{p_2}\fai(x,\,t)\,dx}\r]\fai(B,\,t)\nonumber\\
&&\nonumber\ls \frac{\lz^{p_2}}{\az^{p_2}} \lf\|\chi_B\r\|^{-p_2}_{L^\fai(\rn)}
\fai(B,\,t).
\end{eqnarray}

On the other hand, by Definition \ref{def MOH atom} again, we conclude that
\begin{eqnarray*}
\lf\|\lf\|\chi_B\r\|_{L^\fai(\rn)}\lf[\fai(B,\,t)\r]^{-\frac{1}{p_1}}
a\r\|_{L^{q}_{\fai(\cdot,\,t)}(\rn)}&&=\lf\|\chi_B\r\|_{L^\fai(\rn)}
\lf[\fai(B,\,t)\r]^{\frac{1}{q}-\frac{1}{p_1}} \frac{\|a\|_{L^{q}_{\fai(\cdot,\,t)}(\rn)}}
{\lf[\fai(B,\,t)\r]^{1/{q}}}\\
&&\le \lf[\fai(B,\,t)\r]^{\frac{1}{q}-\frac{1}{p_1}},
\end{eqnarray*}
which immediately implies that
$\|\chi_B\|_{L^\fai(\rn)}[\fai(B,\,t)]^{-\frac{1}{p_1}}a$ is
a weighted $(p_1,\,q,\,s)$-atom associated with $B$
(see \cite{G79,S-T89} for its definition). This, together with
the fact that $T$ is of weak type
$(H_{\fai(\cdot,\,t)}^{p_1}(\rn),\,L_{\fai(\cdot,\,t)}^{p_1}(\rn))$,
implies that
\begin{eqnarray}\label{eqn est2 pro inter}
&&\dint_{\{x\in\rn:\ |T(\lz a)(x)|>\az\}}\fai(x,\,t)\,dx\\
&&\nonumber\hs
= \dint_{\{x\in\rn:\ |T(\lz  \|\chi_B\|_{L^\fai(\rn)}[\fai(B,\,t)]^{-1/p_1}a)
(x)|>\az\|\chi_B\|_{L^\fai(\rn)}[\fai(B,\,t)]^{-1/p_1}\}}\fai(x,\,t)\,dx\\
&&\nonumber\hs \ls\frac{\lz^{p_1}}{\az^{p_1}} \lf\|\chi_B\r\|^{-p_1}_{L^\fai(\rn)}
\fai(B,\,t).
\end{eqnarray}

Now, let $R:=\frac{\lz}{\|\chi_B\|_{L^\fai(\rn)}}$. From the fact that,
for all $(x,t)\in\rn\times(0,\fz)$,
$\fai(x,t)\sim\int_0^t\frac{\fai(x,s)}{s}\,ds$ and Fubini's theorem,
we deduce that
\begin{eqnarray}\label{eqn est3 pro inter}
\dint_{\rn} \fai\lf(x,\,T(\lz a)(x)\r)\,dx &&\sim \dint_0^\fz
\frac{1}{t} \dint_{\{x\in\rn:\ |T(\lz a)(x)|>t\}}\fai(x,\,t)\,dx\,dt\\
&&\nonumber\sim \dint_0^R \frac{1}{t}
\dint_{\{x\in\rn:\ |T(\lz a)(x)|>t\}}\fai(x,\,t)\,dx\,dt+\dint_R^\fz\cdots
\\&&\nonumber=:\mathrm{I}+\mathrm{II}.
\end{eqnarray}

For $\mathrm{I}$, taking $\epsilon\in(0,\,\fz)$ sufficiently small so that
$\frac{\fai(x,\,t)}{t^{p_1+\epsilon}}$ is increasing in $t$,
by using \eqref{eqn est2 pro inter} and the fact $R=\frac{\lz}{\|\chi_B\|_{L^\fai(\rn)}}$,
we see that
\begin{eqnarray}\label{eqn est4 pro inter}
\qquad\mathrm{I}&&\ls \dint_0^R \frac{\lz^{p_1}}{t^{1+p_1}}
\,\|\chi_B\|^{-p_1}_{L^\fai(\rn)}\dint_{B}\fai(x,\,t)\,dx\,dt\\
&&\nonumber\ls\dint_0^R \frac{\lz^{p_1}}{t^{1-\epsilon}}\,dt
\,\|\chi_B\|^{-p_1}_{L^\fai(\rn)}\frac{1}{R^{p_1+\epsilon}}\dint_{B}
\fai(x,\,R)\,dx\ls\dint_{B}
\fai\lf(x,\,\frac{\lz}{\|\chi_B\|_{L^\fai(\rn)}}\r)\,dx.
\end{eqnarray}

Similarly, choosing  $\epsilon\in(0,\,1)$ sufficiently small such that
$\frac{\fai(x,\,t)}{t^{p_2-\epsilon}}$ is decreasing in $t$, it follows,
from \eqref{eqn est1 pro inter}, that
\begin{eqnarray*}
\mathrm{II}&&\ls \dint_R^\fz \frac{\lz^{p_2}}{t^{1+p_2}}
\|\chi_B\|^{-p_2}_{L^\fai(\rn)}\int_{B}\fai(x,\,t)\,dx\,dt\ls\dint_{B}
\fai\lf(x,\,\frac{\lz}{\|\chi_B\|_{L^\fai(\rn)}}\r)\,dx,
\end{eqnarray*}
which, together with \eqref{eqn est3 pro inter} and \eqref{eqn est4 pro inter}, implies
that \eqref{eqn sufficient CD} of Lemma \ref{lem sufficent COMHS} holds true.
This, combined with Lemma \ref{lem sufficent COMHS},
finishes the proof of Proposition \ref{pro interpolation} when (i) holds true.

The proof of the case when (ii) holds true is similar, the details being omitted here.
This finishes the proof of Proposition \ref{pro interpolation}.
\end{proof}

\begin{corollary}\label{cor bd of Riesz}
Let $\fai$ satisfy Assumption $(\fai)$. Then, for all $j\in\{1,\,\ldots,\,
n\}$, the Riesz transform $R_j$ is bounded on $H_\fai(\rn)$.
\end{corollary}

\begin{proof}
Let $\phi\in \mathcal{S}(\rn)$ satisfy $\int_\rn \phi(x)\,dx=1$. For all
$j\in\{1,\,\ldots,\,n\}$, let $T_j:=\mathcal{M}_\phi\circ R_j$,
where $\mathcal{M}_\phi$ is as in \eqref{eqn radial mf}.
Using Proposition \ref{pro radial mc} and the fact that, for all $p\in(0,\,1]$ and
$w\in A_\fz(\rn)$, $R_j$ is bounded on the weighted Hardy space $H^p_w(\rn)$
(see \cite[Theorem 1.1]{Ky11}),
we conclude that $T_j$ is bounded from $H^p_w(\rn)$
to $L^p_w(\rn)$. In particular, let $p_1\in(0,\,i(\fai))$, since, for all $t\in(0,\,\fz)$,
$\fai(\cdot,\,t)\in A_\fz(\rn)$, we know that $T_j$ is bounded from
$H^{p_1}_{\fai(\cdot,\,t)}(\rn)$ to $L^{p_1}_{\fai(\cdot,\,t)}(\rn)$.

On the other hand, let $q(\fai)$ be as in \eqref{1.7} and $p_2\in(q(\fai),\,\fz)$.
From \cite[p.\,411, Theorem 3.1]{GR85}, we deduce that,
for all $w\in A_{p_2}(\rn)$, $R_j$ is bounded on the weighted Lebesgue space
$L^{p_2}_w(\rn)$. Since $\fai(\cdot,\,t)\in A_{p_2}(\rn)$,
we know $R_j$ is bounded on $L^{p_2}_{\fai(\cdot,\,t)}(\rn)$,
which, together with the boundedness of $\mathcal{M}_\phi$ on
$L^{p_2}_{\fai(\cdot,\,t)}(\rn)$, implies that $T_j$ is bounded on
$L^{p_2}_{\fai(\cdot,\,t)}(\rn)$.
Hence, using Propositions \ref{pro radial mc} and \ref{pro interpolation}(i),
we conclude
\begin{eqnarray*}
\|R_j(f)\|_{H_\fai(\rn)}\sim \|T_j(f)\|_{L^\fai(\rn)}\ls \|f\|_{H_\fai(\rn)},
\end{eqnarray*}
which completes the proof of Corollary \ref{cor bd of Riesz}.
\end{proof}

We point out that Proposition \ref{pro interpolation} can also be applied
to the boundedness of Calder\'on-Zygmund operators on $H_\fai(\rn)$.
Recall the following notion of $\tz$-Calder\'on-Zygmund operators
from Yabuta \cite{Ya85}. Let $\tz$ be a nonnegative nondecreasing function on
$(0,\,\fz)$ satisfying $\int_0^1\frac{\tz(t)}{t}\,dt < \fz$. A continuous function
$K: (\rn\times \rn) \setminus \{(x,x) : x \in\rn\} \to \cc$ is called a
$\tz$-\emph{Calder\'on-Zygmund kernel},
if there exists a positive constant $C$ such that, for all $x,\,y\in\rn$ with $x\ne y$,
\begin{eqnarray*}
|K(x,y)|\le \frac{C}{|x - y|^n}
\end{eqnarray*}
and, for all $x,\,x',\,y\in\rn$ with $2|x-x'|<|x-y|$,
\begin{eqnarray*}
|K(x,y)- K(x',\,y)| + |K(y,x) - K(y,x')|\le
\frac{C}{|x -y|^n}\, \tz\lf(\frac{|x-x'|}{|x-y|}\r).
\end{eqnarray*}
A linear operator $T : \mathcal{S}(\rn) \to\mathcal{S}'(\rn)$ is called a
$\tz$-\emph{Calder\'on-Zygmund operator}, if T can be extended to a
bounded linear operator
on $L^2(\rn)$ and there exists a $\tz$-Calder\'on-Zygmund kernel $K$
such that, for all $f\in C_c^\fz(\rn)$ and $x \notin \supp f$,
\begin{eqnarray*}
T f(x)=\dint_\rn K(x,\,y)f(y)\,dy.
\end{eqnarray*}

Recall also that a nonnegative locally integrable function
 $w$ on $\rn$ is said to satisfy the \emph{reverse H\"older condition} for some
$q\in(1,\fz)$, denoted by $w\in \mathrm{RH}_q(\rn)$, if there exists a positive
constant $C$ such that, for all balls $B\subset \rn$,
\begin{eqnarray*}
\lf\{\frac{1}{|B|}\int_B [w(x)]^q\,dx\r\}^{1/q}\le \frac{C}{|B|}\int_B
w(x)\,dx.
\end{eqnarray*}

\begin{corollary}\label{cor bd of CZ}
Let $\dz\in(0,\,1]$, the function $\fai$ satisfy Assumption $(\fai)$,
$q\in[1,\,\frac{i(\fai)(n+\dz)}{n})$, $r\in (\frac{n+\dz}{n+\dz-nq},\,\fz)$ and, for all
$t\in(0,\,\fz)$, $\fai(\cdot,\,t)\in A_q(\rn)\cap \mathrm{RH}_r(\rn)$,
where $i(\fai)$ and $q(\fai)$ are as in \eqref{1.6}
and \eqref{1.7}, respectively.
Assume also that $\tz$ is a nondecreasing function on $[0,\,\fz)$ satisfying
$\int_0^\fz \frac{\tz(t)}{t^{1+\dz}}\,dt<\fz$. If $T$ is a $\tz$-Caldr\'on-Zygmund operator
satisfying $T^*1=0$, namely,  for all $f\in L^\fz(\rn)$ with compact support and
$\int_\rn f(x)\,dx=0$,
\begin{eqnarray*}
\dint_\rn Tf(x)\,dx=0,
\end{eqnarray*}
then $T$ is bounded on $H_\fai(\rn)$.
\end{corollary}

\begin{proof}
To prove Corollary \ref{cor bd of CZ}, recall, in \cite[Theorem 1.2]{Ky11}, that
Ky proved that, for all $\dz\in(0,\,1]$, $p_1\in(\frac{n}{n+\dz},\,1]$,
$q\in[1,\,\frac{p_1(n+\dz)}{n})$,
$r\in(\frac{n+\dz}{n+\dz-nq},\,\fz)$ and $w\in A_q(\rn)\cap \mathrm{RH}_r(\rn)$,
the $\tz$-Calder\'on-Zygmund operator $T$, with $\tz$ satisfying the same
assumptions as in this corollary, is bounded on the weighted Hardy space
$H_w^{p_1}(\rn)$,  if $T^*1= 0$. In particular, let $p_1\in(\frac{n}{n+\dz},\,i(\fai))$, we know
$q\in[1,\,\frac{i(\fai)(n+\dz)}{n})$ and $r\in (\frac{n+\dz}{n+\dz-nq},\,\fz)$. Thus,
for all $t\in(0,\,\fz)$, $\fai(\cdot,\,t)\in  A_q(\rn)\cap \mathrm{RH}_r(\rn)$
and hence $T$ is bounded on $H^{p_1}_{\fai(\cdot,\,t)}(\rn)$, if $T^*1= 0$.

On the other hand, let $q(\fai)$ be as in \eqref{1.7}. From \cite[Theorem 2.4]{Ya85},
we deduce that, for all $p_2\in(q(\fai),\,\fz)$ and $w\in A_{p_2}(\rn)$,  $T$ is bounded on $L_w^{p_2}(\rn)$. Since, $p_2>q(\fai)$, we know that, for all $t\in(0,\,\fz)$,
$\fai(\cdot,\,t)\in  A_{p_2}(\rn)$. Thus, $T$ is bounded on $L_{\fai(\cdot,\,t)}^{p_2}(\rn)$.
Moreover, let $\mathcal{M}_\phi$ be as in \eqref{eqn radial mf} and
$S:=\mathcal{M}_\phi\circ T$. Using Proposition \ref{pro radial mc} and
the boundedness of $\mathcal{M}_\phi$ on $L^{p_2}_{\fai(\cdot,\,t)}(\rn)$,
we conclude that, for all $t\in(0,\,\fz)$,
$S$ is bounded from
$H^{p_1}_{\fai(\cdot,\,t)}(\rn)$ to $L^{p_1}_{\fai(\cdot,\,t)}(\rn)$ and
bounded on $L^{p_2}_{\fai(\cdot,\,t)}(\rn)$.
By Proposition \ref{pro interpolation}(i), we know
that $S$ is bounded from $H_\fai(\rn)$ to $L^\fai(\rn)$. This, together with Proposition
\ref{pro radial mc}, implies that $T$ is bounded on $H_\fai(\rn)$, which completes the proof
of Corollary \ref{cor bd of CZ}.
\end{proof}

\begin{remark}\label{rem bd of many CO}
We point out that it is well known that many important operators are bounded
on weighted Hardy spaces. Thus, by using Proposition \ref{pro interpolation}, we can
obtain their corresponding boundedness on $H_\fai(\rn)$.
\end{remark}

We now turn to the proof of Theorem \ref{thm1}.

\begin{proof}[Proof of Theorem \ref{thm1}]
We prove Theorem \ref{thm1} by showing that
\begin{eqnarray}\label{eqn est1 thm1}
\lf(H_\fai(\rn)\cap L^2(\rn)\r)=\mathbb{H}_{\fai,\,\mathrm{Riesz}}(\rn)
\end{eqnarray}
with equivalent quasi-norms.

We first show the inclusion that $(H_\fai(\rn)\cap L^2(\rn))
\subset \mathbb{H}_{\fai,\,\mathrm{Riesz}}(\rn)$.
Let $f\in H_\fai(\rn)\cap L^2(\rn)$ and $\phi\in \mathcal{S}(\rn)$
satisfy \eqref{eqn integral=1}.
By Proposition \ref{pro radial mc} and Corollary
\ref{cor bd of Riesz}, we see that
\begin{eqnarray}\label{3.x2}
\lf\|f\r\|_{{H}_{\fai,\,\mathrm{Riesz}}(\rn)}&&=\lf\|f\r\|_{L^\fai(\rn)}
+\dsum_{j=1}^n\lf\|R_j(f)\r\|_{L^\fai(\rn)}\\
&&\nonumber\le \lf\|\mathcal{M}_\phi(f)\r\|_{L^\fai(\rn)}
+\dsum_{j=1}^n\lf\|\mathcal{M}_\phi(R_j(f))\r\|_{L^\fai(\rn)}\\
&&\nonumber\sim \lf\|f\r\|_{H_\fai(\rn)}
+\dsum_{j=1}^n\lf\|R_j(f)\r\|_{H_\fai(\rn)}\ls \lf\|f\r\|_{H_\fai(\rn)},
\end{eqnarray}
where $\mathcal{M}_\phi$ denotes the radial maximal function as in
\eqref{eqn radial mf}.
This implies that $f\in \mathbb{H}_{\fai,\,\mathrm{Riesz}}(\rn)$ and hence the
inclusion $(H_\fai(\rn)\cap L^2(\rn)) \subset
\mathbb{H}_{\fai,\,\mathrm{Riesz}}(\rn)$ holds true.

We now turn to the proof of the inclusion $\mathbb{H}_{\fai,\,\mathrm{Riesz}}(\rn)
\subset (H_\fai(\rn)\cap L^2(\rn))$. Let $f\in \mathbb{H}_{\fai,\,\mathrm{Riesz}}(\rn)$.
For all $(x,\,t)\in\rr^{n+1}_+$, let
\begin{eqnarray*}
F(x,\,t)&&:=(u_0(x,\,t),\,u_1(x,\,t),\,\ldots,\,u_n(x,\,t))\\
&&:=\lf((f*P_t)(x),\, (f*Q^{1}_t)(x),\,\ldots,\,
(f*Q^{n}_t)(x)\r),
\end{eqnarray*}
where $P_t$ is the Poisson kernel as in \eqref{eqn PK} and, for all
$j\in\{1,\,\ldots,\,n\}$, $Q^{(j)}_t$ is the conjugant Poisson kernel
as in \eqref{eqn CPK}. From $f\in L^2(\rn)$ and \cite{Ho53}
(see also \cite[p.\,78, 4.4]{St70}), we deduce that
$F$ satisfies the generalized Cauchy-Riemann equation \eqref{eqn GCR equation}.
Thus, we know that, for $q\in[\frac{n-1}{n},\,\frac{i(\fai)}{q(\fai)})$,
$|F|^{q}$ is subharmonic (see, for example,
\cite[p.\,234, Theorem 4.14]{SW71}).
Moreover, by \cite[p.\,80, Theorem 4.6]{SW71},
we obtain the following harmonic majorant that, for all $(x,\,t)\in\rr^{n+1}_+$,
\begin{eqnarray*}
\lf|F(x,\,t)\r|^q\le \lf(\lf|F(\cdot,\,0)\r|^q*P_t\r)(x),
\end{eqnarray*}
where $F(\cdot,\,0)=\{f,\,R_1(f),\,\ldots,\,R_n(f)\}$
via the Fourier transform.
Thus, it follows, from \eqref{eqn scaling of MOF} and Lemma \ref{lem bd HLMF}, that
\begin{eqnarray*}
\dsup_{t\in(0,\,\fz)} \lf\|\lf|F(\cdot,\,t)\r|\r\|_{L^\fai(\rn)} &&=\dsup_{t\in(0,\,\fz)} \lf\|\lf|F(\cdot,\,t)\r|^q\r\|^{1/q}_{L^{\fai_q}(\rn)} \le \dsup_{t\in(0,\,\fz)} \lf\|\mathcal{M}\lf(\lf|F(\cdot,\,0)\r|^q\r)\r\|^{1/q}_{L^{\fai_q}(\rn)}\\
&&\ls \dsup_{t\in(0,\,\fz)} \lf\|\lf|F(\cdot,\,0)\r|\r\|_{L^{\fai}(\rn)}\ls
\lf\|f\r\|_{L^\fai(\rn)}+\dsum_{j=1}^n\lf\|R_j(f)\r\|_{L^\fai(\rn)}\\
&&\sim \|f\|_{H_{\mathbb{H}_{\fai,\,\mathrm{Riesz}}(\rn)}},
\end{eqnarray*}
where $\fai_q$ is as in \eqref{2.y2} and $\mathcal{M}$ the Hardy-Littlewood
maximal function as in \eqref{HL maxiamal function}.
Thus, $F\in \mathcal{H}_\fai(\rr^{n+1}_+)$ and
\begin{eqnarray*}
\|F\|_{\mathcal{H}_\fai(\rr^{n+1}_+)}
\ls \|f\|_{H_{\mathbb{H}_{\fai,\,\mathrm{Riesz}}(\rn)}}.
\end{eqnarray*}
Moreover, from $f\in L^2(\rn)$ and
\cite[Theorem 4.17(i)]{SW71}, we further deduce $F\in \mathcal{H}^2(\rr^{n+1}_+)$,
which, together with $F\in \mathcal{H}_\fai(\rr^{n+1}_+)$ and
Theorem \ref{cor equiv 3spaces},  further implies that
$f\in H_\fai(\rn)$ and
$\|f\|_{{H}_\fai(\rn)}\ls \|f\|_{H_{\mathbb{H}_{\fai,\,\mathrm{Riesz}}(\rn)}}$.
Thus, $f\in H_\fai(\rn)\cap L^2(\rn)$,
which completes the proof of Theorem \ref{thm1}.
\end{proof}

\section{Higher order Riesz transform characterizations}
\label{s3}

\hskip\parindent
In this section, we give out the proofs of Theorems \ref{thm2} and \ref{thm3}.
First, we introduce the Musielak-Orlicz-Hardy spaces
$\mathcal{H}_{\fai,\,m}(\rr^{n+1}_+)$ of tensor-valued functions of rank $m$ with
$m\in\nn$.

Let $n,\,m\in\nn$ and $\{e_0,\,e_1,\,\ldots,\,e_n\}$ be an orthonormal basis
of $\rr^{n+1}$. The \emph{tensor product  of $m$ copies
of} $\rr^{n+1}$ is defined to be the set
\begin{eqnarray*}
\bigotimes^m \rr^{n+1}:=\lf\{F:=\dsum_{j_1,\,\ldots,\,j_m=0}^nF_{j_1,\,\ldots,\,j_m}\,
e_{j_1}\otimes \cdots \otimes e_{j_m}:\  \ F_{j_1,\,\ldots,\,j_m}\in\cc\r\},
\end{eqnarray*}
where $e_{j_1}\otimes \cdots \otimes e_{j_m}$ denotes the \emph{tensor product} of
$e_{j_1},\, \ldots,\,e_{j_m}$ and  each $F\in \bigotimes\limits^m \rr^{n+1}$ is called
a \emph{tensor of rank} $m$.

Let $F:\ \rr^{n+1}_+\to \bigotimes\limits^m \rr^{n+1}$ be a
tensor-valued function of rank $m$ of the form that, for all $(x,\,t)\in\rr^{n+1}_+$,
\begin{eqnarray}\label{3.x1}
F(x,\,t)=\dsum_{j_1,\,\ldots,\,j_m=0}^nF_{j_1,\,\ldots,\,j_m}(x,\,t)\,
e_{j_1}\otimes \cdots \otimes e_{j_m}
\end{eqnarray}
with $F_{j_1,\,\ldots,\,j_m}(x,\,t)\in\cc$. Then the tensor-valued function 
$F$ of rank $m$ is said to be \emph{symmetric},
if, for any permutation $\sz$ on $\{1,\,\ldots,\,m\}$,
$j_1,\,\ldots,\,j_m\in\{0,\,\dots,\,n\}$ and $(x,\,t)\in\rr^{n+1}_+$,
\begin{eqnarray*}
F_{j_1,\,\ldots,\,j_m}(x,\,t)=F_{j_{\sz(1)},\,\ldots,\,j_{\sz(m)}}(x,\,t).
\end{eqnarray*}
For $F$ being symmetric, $F$ is said to be of \emph{trace zero} if, for all
$j_3,\,\ldots,\,j_m\in\{0,\,\dots,\,n\}$ and $(x,\,t)\in\rr^{n+1}_+$,
\begin{eqnarray*}
\dsum_{j=0}^n F_{j,\,j,\,j_3,\,\ldots,\,j_m}(x,\,t)\equiv 0.
\end{eqnarray*}

Let $F$ be as in \eqref{3.x1}. Its \emph{gradient} $\nabla F: \ \rr^{n+1}_+\to
\bigotimes\limits^{m+1} \rr^{n+1}$
is a tensor-valued function of rank $m+1$ of the form that, for all $(x,\,t)\in\rr^{n+1}_+$,
\begin{eqnarray*}
\nabla F(x,\,t)&&=\dsum_{j=0}^n \frac{\pat F}{\pat x_j}(x,\,t)\otimes e_j\\
&&=\dsum_{j=0}^n  \dsum_{j_1,\,\ldots,\,j_m=0}^n\frac{\pat F_{j_1,\,\ldots,\,j_m}}
{\pat x_j}(x,\,t)\, e_{j_1}\otimes \cdots \otimes e_{j_m}\otimes e_j,
\end{eqnarray*}
here and hereafter, we always let $x_0:=t$. A tensor-valued function
$F$ is said to satisfy the \emph{generalized Cauchy-Riemann equation},
if both $F$ and $\nabla F$ are symmetric and of trace zero. We point out that, if
$m=1$, this definition of generalized Cauchy-Riemann equations is equivalent to
the generalized Cauchy-Riemann equation as in \eqref{eqn GCR equation}.
For more details on the generalized Cauchy-Riemann equation on tensor-valued
functions, we refer the reader to \cite{SW68,PS08}.

The following is a generalization of Musielak-Orlicz-Hardy spaces
$\mathcal{H}_\fai(\rr^{n+1}_+)$ of harmonic vectors
defined in Definition \ref{def MOHVH space}.

\begin{definition}\label{def MOTVFH}
Let $m\in\nn$ and $\fai$ satisfy Assumption $(\fai)$.
The \emph{Musielak-Orlicz-Hardy space} $\mathcal{H}_{\fai,\,m}(\rr^{n+1}_+)$
\emph{of tensor-valued functions of rank} $m$ is defined to be the set of all tensor-valued
functions $F$, of rank $m$, satisfying the generalized Cauchy-Riemann equation.
For any $F\in\mathcal{H}_{\fai,\,m}(\rr^{n+1}_+)$, its \emph{quasi-norm} is defined by
\begin{eqnarray*}
\|F\|_{\mathcal{H}_{\fai,\,m}(\rr^{n+1}_+)}:=
\dsup_{t>0}\lf\|\lf|F(\cdot,\,t)\r|\r\|_{L^\fai(\rn)},
\end{eqnarray*}
where, for all $(x,\,t)\in\rr^{n+1}_+$,
\begin{eqnarray*}
|F(x,\,t)|:=\lf\{\dsum_{j_1,\,\ldots,\,j_m=0}^n \lf|{F_{j_1,\,\ldots,\,j_m}}
(x,\,t)\r|^2\r\}^{1/2}.
\end{eqnarray*}
\end{definition}

Moreover, Stein and Weiss \cite{SW68} proved the following result.

\begin{proposition}[\cite{SW68}]\label{pro SHP of TVF}
Let $m\in\nn$ and $F$ be a tensor-valued
functions of rank $m$ satisfying the generalized Cauchy-Riemann equation.
Then, for all $p\in[\frac{n-1}{n+m-1},\,\fz)$, $|F|^p$ is subharmonic on $\rr^{n+1}_+$.
\end{proposition}

Recall also the following result from Calder\'on and Zygmund \cite[Theorem 1]{CZ64}.

\begin{proposition}[\cite{CZ64}]\label{pro SHP of HGHF}
Let $m\in\nn$ and $u$ be a harmonic function on $\rr^{n+1}_+$. For all $p\in[\frac{n-1}
{n+m-1},\,\fz)$, $|\nabla^m u|^p$ is subharmonic. Here, for all $(x,\,t)\in\rr^{n+1}_+$,
\begin{eqnarray*}
\nabla^m u(x,\,t):=\lf\{\pat^{\az} u(x,\,t)\r\}_{|\az|=m}
\end{eqnarray*}
with $\az:=\{\az_0,\,\ldots,\,\az_n\}\in\zz_+^{n+1}$, $|\az|:=\sum_{j=0}^{n}
|\az_j|$, $x_0:=t$ and $\pat^{\az}:=(\frac{\pat}{\pat x_0})^{\az_0}\cdots
(\frac{\pat}{\pat x_n})^{\az_n}$.
\end{proposition}

It is known that every harmonic vector satisfying the generalized Cauchy-Riemann
equation \eqref{eqn GCR equation} is a gradient of a harmonic function on $\rr^{n+1}_+$.
A similar result still holds true for tensor-valued functions, which is the following
proposition.

\begin{proposition}[\cite{SW71,Uc01}]\label{pro Relation TVF&HG}
Let $m\in\nn$ with $m\ge 2$, $F$ be a tensor-valued
function of rank $m$ satisfying that both $F$ and $\nabla F$ are symmetric,
and $F$ is of trace zero. Then there exists a harmonic function $u$ on $\rr^{n+1}_+$
such that $\nabla^mu=F$, namely, for all $\{j_1,\,\ldots,\,j_m\}\subset\{0,\,1,\,\ldots,\,
n\}$ and $(x,\,t)\in\rr^{n+1}_+$,
\begin{eqnarray*}
\frac{\pat}{\pat x_{j_1}}\cdots \frac{\pat}{\pat x_{j_m}} u(x,\,t)
=F_{j_1,\,\cdots,\,j_m}(x,\,t).
\end{eqnarray*}
\end{proposition}

\begin{remark}\label{rem about the ToG}
(i) Propositions \ref{pro SHP of HGHF} and \ref{pro Relation TVF&HG} imply that,
if $m\ge 2$, then the condition that $\nabla F$ has trace zero, in the generalized
Cauchy-Riemann equation, can be removed to ensure that Proposition
\ref{pro SHP of TVF} still holds true.

(ii)  We also point out that, in Proposition \ref{pro HVC},
Lemmas \ref{lem harmonic majorant of MOHVH} and \ref{lem boudnary value of MOHVH},
and Corollary \ref{cor equiv 3spaces},
we used the restriction that $\frac{i(\fai)}{q(\fai)}>\frac{n-1}{n}$, only because, for all
$p\in[\frac{n-1}{n},\,\fz)$, the $p$-power of the absolute value of the first-order gradient
$|\nabla u|^p$ of a harmonic function on $\rr^{n+1}_+$ is subharmonic.
Since, for all $m\in\nn$ and $p\in[\frac{n-1}{n+m-1},\,\fz)$, $|\nabla^m u|^p$ is subharmonic
on $\rr^{n+1}_+$, the restriction $\frac{i(\fai)}{q(\fai)}>\frac{n-1}{n}$ can be relaxed to
$\frac{i(\fai)}{q(\fai)}>\frac{n-1}{n+m-1}$, with the Musielak-Orlicz-Hardy space
$\mathcal{H}_\fai(\rr^{n+1}_+)$ of harmonic vectors replaced by the Musielak-Orlicz-Hardy space
$\mathcal{H}_{\fai,\,m}(\rr^{n+1}_+)$ of tensor-valued functions of rank $m$.
Moreover, for any given Musielak-Orlicz function $\fai$ satisfying Assumption $(\fai)$,
by letting $m$ be sufficiently large, we know that Proposition \ref{pro HVC},
Lemmas \ref{lem harmonic majorant of MOHVH} and \ref{lem boudnary value of MOHVH},
and Theorem \ref{cor equiv 3spaces} always hold true for
$\frac{i(\fai)}{q(\fai)}>\frac{n-1}{n+m-1}$.
\end{remark}

Now we give out the proof of Theorem \ref{thm2}.

\begin{proof}[Proof of Theorem \ref{thm2}]
The proof of Theorem \ref{thm2} is similar to that of
Theorem \ref{thm1}. In particular, the second inequality of \eqref{1.9}
is an easy consequence of Proposition \ref{pro radial mc}
and Corollary \ref{cor bd of Riesz}. Indeed, let $f\in H_\fai(\rn)\cap
L^2(\rn)$. By Proposition \ref{pro radial mc}, Corollary \ref{cor bd of Riesz}
and an argument similar to that used in \eqref{3.x2}, we see that
\begin{eqnarray*}
\lf\|f\r\|_{L^\fai(\rn)}+
\dsum_{k=1}^m \dsum_{j_1,\,\ldots,\,j_k=1}^n \lf\|R_{j_1}
\cdots R_{j_k}(f)\r\|_{L^\fai(\rn)}\ls\|f\|_{H_\fai(\rn)},
\end{eqnarray*}
which implies the second inequality of \eqref{1.9}.

To prove the first inequality of \eqref{1.9}, let $f\in L^2(\rn)$ satisfy \eqref{1.9x}.
We construct the tensor-valued function $F$
of rank $m$ by setting, for all $\{j_1,\,\ldots,\,j_m\}\subset \{0,\,\ldots,\,n\}$ and
$(x,\,t)\in\rr^{n+1}_+$,
\begin{eqnarray*}
F_{j_1,\,\ldots,\,j_m}(x,\,t):=\lf(\lf(R_{j_1}\cdots R_{j_m}(f)\r)*P_t\r)(x),
\end{eqnarray*}
where $P_t$ is the Poisson kernel as in \eqref{eqn PK} and $R_0:=I$ denotes
the identity operator.
We know $F:=\sum_{j_1,\,\ldots,\,j_m=0}^n
F_{j_1,\,j_2,\,\ldots,\,j_m}\,
e_{j_1}\otimes \cdots \otimes e_{j_m}$
satisfies the generalized Cauchy-Riemann equation via the Fourier transform (see also
the proof of \cite[Lemma 17.1]{Uc01}).
Also, a corresponding harmonic majorant holds true (see also \cite[Lemma 17.2]{Uc01}),
namely, for all $q\in [\frac{n-1}{n+m-1},\,\frac{i(\fai)}{q(\fai)})$
and $(x,\,t)\in\rr^{n+1}_+$,
\begin{eqnarray*}
\lf|F(x,\,t)\r|^q\le \lf(\lf|F(x,\,0)\r|^q*P_t\r)(x),
\end{eqnarray*}
where $F(x,\,0):=\{R_{j_1}\cdots R_{j_m}(f)(x)\}_{\{j_1,\ldots,\,j_m\}
\subset \{0,\,\ldots,\,n\}}$,
which, combined with \eqref{eqn scaling of MOF},
\eqref{1.9x} and Lemma \ref{lem bd HLMF}, implies that
$F\in \mathcal{H}_{\fai,\,m}(\rr^{n+1}_+)$ and
\begin{eqnarray}\label{eqn est1 thm HRTC-1}
\lf\|F\r\|_{\mathcal{H}_{\fai,\,m}(\rr^{n+1}_+)}&&
=\dsup_{t\in(0,\,\fz)}\lf\|\lf|F(\cdot,\,t)\r|\r\|_{L^\fai(\rn)}
\ls \lf\|\mathcal{M}\lf(\lf|F(\cdot,\,0)\r|^q\r)\r\|^{1/q}_{L^{\fai_q}(\rn)}\\
&&\nonumber
\ls \dsum_{j_1,\,\ldots,\,j_m=0}^n
\lf\|R_{j_1}\cdots R_{j_m}(f)\r\|_{L^\fai(\rn)}\\
\nonumber&&
\ls \lf\| f\r\|_{L^\fai(\rn)}
+\dsum_{k=1}^m \dsum_{j_1,\,\ldots,\,j_k=1}^n
\lf\|R_{j_1}\cdots R_{j_k}(f)\r\|_{L^\fai(\rn)}\ls A,
\end{eqnarray}
where $\mathcal{M}$ denotes the Hardy-Littlewood maximal function as in \eqref{HL maxiamal function}.
This, together with Remark \ref{rem about the ToG}(ii)
(a counterpart to Theorem \ref{cor equiv 3spaces}), implies that
$f\in H_\fai(\rn)$
and the first inequality of \eqref{1.9} holds true, which completes the proof of
Theorem \ref{thm2}.
\end{proof}

We now turn to the proof of Theorem \ref{thm3}.
To this end, we recall some facts on Fourier multipliers.

Let $f\in \mathcal{S}(\rn)$, $\mathbb{S}^{n-1}$ be the \emph{unit sphere} in $\rn$
and $\tz\in L^\fz(\mathbb{S}^{n-1})$. The \emph{Fourier multiplier} $K$ of $f$ with the
\emph{multiplier function} $\tz$ is defined by setting, for all $\xi\in\rn$,
\begin{eqnarray*}
K(f)(\xi):=\mathcal{F}^{-1}\lf(\tz\lf(\frac{\cdot}{|\cdot|}\r)\mathcal{F}(f)(\cdot)\r)
(\xi),
\end{eqnarray*}
where $\mathcal{F}$ and $\mathcal{F}^{-1}$ denote, respectively, the
\emph{Fourier transform} and its \emph{inverse}.

It is easy to see that, for all $j\in\{1,\,\ldots,\,n\}$, the Riesz transform $R_j$
is a Fourier multiplier with the multiplier function $\tz_j(\xi):=-i{\xi_j}$ for all
$\xi\in \mathbb{S}^{n-1}$. Also, for all $k\in\nn$ and
$\{j_1,\,\ldots,\,j_k\}\subset \{1,\,\ldots,\,n\}$,
the higher Riesz transform $R_{j_1}\cdots R_{j_k}$ is also a Fourier multiplier
with the multiplier function that, for all $\xi\in\mathbb{S}^{n-1}$,
\begin{eqnarray}\label{eqn multiplier HRT}
\tz_{j_1,\,\ldots,\,j_k}(\xi):=\lf(-i{\xi_{j_1}}\r)
\cdots\lf(-i{\xi_{j_k}}\r).
\end{eqnarray}

\begin{proposition}\label{pro bd FourierM}
Let $\fai$ satisfy Assumption $(\fai)$ and
$\tz\in C^\fz(\mathbb{S}^{n-1})$. Then the Fourier multiplier $K$ with the multiplier
function $\tz$ is bounded on $H_\fai(\rn)$.
\end{proposition}

\begin{proof}
By $\tz\in C^\fz(\mathbb{S}^{n-1})$, we deduce, from \cite[p.\,176, Theorem 14]{S-T89},
that, for all $p_1\in(0,\,1]$ and $w\in A_\fz(\rn)$, $K$ is bounded on the weighted Hardy
space $H_w^{p_1}(\rn)$ and that, for all $s\in(1,\,\fz)$, $w\in A_s(\rn)$
and $p_2\in (2s,\,\fz)$,  $K$ is bounded on
$L_w^{p_2}(\rn)$, which, together with Propositions
\ref{pro radial mc} and \ref{pro interpolation}, and an argument similar to
that used in the proofs of Corollaries \ref{cor bd of Riesz} and \ref{cor bd of CZ},
implies that $K$ is bounded on $H_\fai(\rn)$.
This finishes the proof of Proposition \ref{pro bd FourierM}.
\end{proof}

Now, let $m\in\nn$ and $\mathcal{K}:=\{K_1,\,\ldots,\,K_m\}$, where, for each
$j\in\{1,\,\ldots,\,m\}$, $K_j$ is a Fourier multiplier with the multiplier function
$\tz_j\in C^\fz(\mathbb{S}^{n-1})$. For any $f\in L^2(\rn)$, let
\begin{eqnarray}\label{eqn vectot FM}
\mathcal{K}(f):=(K_1(f),\,\ldots,\,K_m(f)).
\end{eqnarray}
For any $q\in(0,\,\fz)$, the $q$-\emph{order maximal function} $\mathcal{M}_q(f)$ of $f$
is defined by setting, for all $x\in\rn$,
\begin{eqnarray}\label{eqn q-order MF}
\mathcal{M}_q(f)(x):=\dsup_{B\ni x}\lf\{\frac{1}{|B|}
\dint_B \lf|f(y)\r|^q\,dy\r\}^{1/q},
\end{eqnarray}
where the supremum is taking over all balls $B$ of $\rn$ containing $x$.
Using Lemma \ref{lem bd HLMF}, we see that, if $i(\fai)>qq(\fai)$, $\mathcal{M}_q$
is bounded on $L^\fai(\rn)$.

Now we recall the following result from Uchiyama \cite[Theorem 2]{Uc84}.

\begin{proposition}[\cite{Uc84}]\label{pro est MFFM}
Let $m\in\nn$, $j\in\{1,\,\ldots,\,m\}$, $\tz_j\in C^\fz(\mathbb{S}^{n-1})$ and
$\mathcal{K}$, having the form $\{K_1,\,\ldots,\,K_m\}$, be a vector of Fourier multipliers
with the multiplier functions of the form $\{\tz_1,\,\ldots,\,\tz_m\}$. If
\begin{eqnarray}\label{eqn rankC}
\mathrm{Rank}
\begin{pmatrix}
\tz_1(\xi), \quad & \ldots,\quad &  \tz_m(\xi) \\
\tz_1(-\xi), \quad & \ldots,\quad &  \tz_m(-\xi)
\end{pmatrix}
\equiv 2
\end{eqnarray}
on $\mathbb{S}^{n-1}$, where $\mathrm{Rank}\,(\cdot)$ denotes of
the rank of a matrix, then there exist $p_0\in(0,\,1)$ and a positive constant
 $C$, depending only on
$\tz_1,\,\ldots,\,\tz_m$, such that, for all $f\in L^2(\rn)$ and $x\in\rn$,
\begin{eqnarray}\label{eqn est MFFM}
\mathcal{M}_\phi\lf(\mathcal{K}(f)\r)(x)\le C \mathcal{M}_{p_0}
\lf(\mathcal{M}_{1/2} \lf(|\mathcal{K}(f)|\r)\r)(x),
\end{eqnarray}
where $\phi\in \mathcal{S}(\rn)$ satisfies \eqref{eqn integral=1},
\begin{eqnarray*}
\mathcal{M}_\phi\lf(\mathcal{K}(f)\r):=\dsup_{t\in(0,\,\fz)}\lf|\lf(K_1(f)*\phi_t,\,\ldots,\,
K_m(f)*\phi_t\r)\r|,
\end{eqnarray*}
$|\mathcal{K}(f)|:=(\sum_{j=1}^m |K_j(f)|^2)^{1/2}$,  $\mathcal{M}_{p_0}$
and $\mathcal{M}_{1/2}$ are as in \eqref{eqn q-order MF}.
\end{proposition}

\begin{remark}\label{rem est MFFM}
(i) Inequality \eqref{eqn est MFFM} provides a good substitute for
the subharmonic property of $|F|^p$ for the harmonic vector (resp. tensor-valued function)
$F$, which enables us to use less Riesz transforms than Theorem \ref{thm2} to characterize
$H_\fai(\rn)$, but at the expense that we do not know the exact value of the
exponent $p_0$ in \eqref{eqn est MFFM}.

(ii) Let $k\in\nn$ and $\mathcal{K}
:=\{I\}\cup \{R_{j_1}\cdots R_{j_k}\}_{j_1,\,\ldots,\, j_k=1}^n$
consist of the identity operator $I$ and all $k$-order Riesz transforms
$R_{j_1}\cdots R_{j_k}$ defined as in Remark \ref{r1.2}(i).
Then we know that
\begin{eqnarray}\label{eqn rankC2}
\mathrm{Rank}
\begin{pmatrix}
&1,\hs &(-i{\xi_1})^k, \hs & \ldots,\hs &  (-i{\xi_n})^k \\
&1,\hs &(-1)^k(-i{\xi_1})^k, \hs & \ldots,\hs &  (-1)^k(-i{\xi_n})^k
\end{pmatrix}
\equiv 2
\end{eqnarray}
on $\mathbb{S}^{n-1}$ if and only if $k$ is odd. Recall that
Gandulfo, Garc\'ia-Cuerva and Taibleson \cite{GGT76}
have constructed a counterexample to show that the even order Riesz transforms
fail to characterize $H^1(\mathbb{R}^2)$. This implies the possibility of  using the
odd order Riesz transforms to characterize the Hardy type spaces.
\end{remark}

Now, we show Theorem \ref{thm3}.

\begin{proof}[Proof of Theorem \ref{thm3}]
The proof of the second inequality of \eqref{1.10} is
an easy consequence of Proposition \ref{pro radial mc} and
Corollary \ref{cor bd of Riesz} (see also the proof of
the second inequality of \eqref{1.9} of Theorem \ref{thm2}),
the details being omitted.

We now turn to the proof of the first inequality of \eqref{1.10}.
Recall that $\tz_{j_1,\,\ldots,\,j_k}$, defined as in \eqref{eqn multiplier HRT}, is the
multiplier function of $R_{j_1}\cdots R_{j_k}$. From \cite[p.\,224]{Uc84} (or
the proof of \cite[p.\,170, Theorem 10.2]{Uc01}), we deduce that
there exists $\{\psi\}\cup\{\psi_{j_1,\,\ldots,\,j_{k}}\}_{j_1,\,\ldots,\,j_k=1}^n\subset
C^\fz(\mathbb{S}^{n-1})$ such that, for all $\xi\in \mathbb{S}^{n-1}$,
\begin{eqnarray*}
\psi(\xi)+\dsum_{j_1,\,\ldots,\,j_k=1}^n \tz_{j_1,\,\ldots,\,j_{k}}(\xi)
\psi_{j_1,\,\ldots,\,j_{k}}(\xi)=1,
\end{eqnarray*}
which, together with Proposition \ref{pro bd FourierM},
\eqref{eqn est MFFM}, $\frac{i(\fai)}{q(\fai)}>\max\{p_0,\,\frac{1}{2}\}$ and the fact that $\mathcal{M}_{q_0}\circ \mathcal{M}_{1/2}$ is bounded on $L^\fai(\rn)$, implies that
\begin{eqnarray*}
\lf\|f\r\|_{H_\fai(\rn)}&&\le \lf\|(\hspace{0.08cm}\psi\widehat{f}\hspace{0.08cm})^\vee\r\|_{H_\fai(\rn)}+
\dsum_{j_1,\,\ldots,\,j_k=1}^n\lf\|(\hspace{0.08cm}\tz_{j_1,\,\ldots,\,j_{k}}
\psi_{j_1,\,\ldots,\,j_{k}}\widehat{f}\hspace{0.08cm})^\vee\r\|_{H_\fai(\rn)}\\
&&\ls  \lf\|f\r\|_{H_\fai(\rn)}+
\dsum_{j_1,\,\ldots,\,j_k=1}^n
\lf\|R_{j_1}\cdots R_{j_k} (f)\r\|_{H_\fai(\rn)}\\
&&\ls \lf\|\mathcal{M}_\phi(\mathcal{K}(f))\r\|_{L^\fai(\rn)}
\ls \lf\|\mathcal{M}_{p_0}
\lf(\mathcal{M}_{1/2} \lf(|\mathcal{K}(f)|\r)\r)\r\|_{L^\fai(\rn)}\\
&&\ls \lf\||\mathcal{K}(f)|\r\|_{L^\fai(\rn)}\ls
 \lf\| f\r\|_{L^\fai(\rn)}
+\dsum_{j_1,\,\ldots,\,j_k=1}^n \lf\| R_{j_1}\cdots R_{j_k}(f)\r\|_{L^\fai(\rn)},
\end{eqnarray*}
where $\mathcal{K}:=\{I\}\cup \{R_{j_1}\cdots R_{j_k}\}_{j_1,\,\ldots,\, j_k=1}^n$,
$I$ is the identity operator,\ $\widehat{}$\ \ and $^\vee$ denote,
respectively, the Fourier transform and its inverse.
This proves the first inequality of \eqref{1.10} and hence finishes the
proof of Theorem \ref{thm3}.
\end{proof}

\medskip

{\bf Acknowledgements.} The authors would like to thank the referee
for her/his careful reading and several valuable remarks which
improve the presentation of this article.

\bigskip

\noindent Jun Cao and Dachun Yang (Corresponding author)

\medskip

\noindent School of Mathematical Sciences, Beijing Normal
University, Laboratory of Mathematics and Complex Systems, Ministry
of Education, Beijing 100875, People's Republic of China

\smallskip

\noindent{\it E-mails:} \texttt{caojun1860@mail.bnu.edu.cn} (J. Cao)

\hspace{1.16cm}\texttt{dcyang@bnu.edu.cn} (D. Yang)

\bigskip

\noindent Der-Chen Chang

\medskip

\noindent Department of Mathematics and Department of Computer Science,
Georgetown University, Washington D.C. 20057, USA

\noindent Department of Mathematics, Fu Jen Catholic University, Taipei 242, Taiwan

\smallskip

\noindent{\it E-mail:} \texttt{chang@georgetown.edu}

\bigskip
\medskip

\noindent Sibei Yang

\medskip

\noindent School of Mathematics and Statistics,
Lanzhou University, Lanzhou, Gansu 730000, People's Republic of China

\smallskip

\noindent{\it E-mails:} \texttt{yangsb@lzu.edu.cn}

\end{document}